\documentclass[11pt]{article}
\usepackage[utf8]{inputenc}
\usepackage{amssymb}
\usepackage{amsmath}
\usepackage{centernot}
\usepackage{amsthm}
\usepackage{color}
\usepackage{hyperref}
\usepackage{titlesec}
\allowdisplaybreaks[1] 
\usepackage[margin=3cm]{geometry}

\titleformat{\subsection}[runin]{\normalfont\itshape}{\thesubsection\hspace{8pt}}{3pt}{}[.] 

\theoremstyle{plain}
\newtheorem{thm}{Theorem}
\newtheorem{lem}[thm]{Lemma}
\newtheorem{cor}[thm]{Corollary}
\newtheorem{prop}[thm]{Proposition}
\newtheorem{defi}[thm]{Definition}
\theoremstyle{remark}
\newtheorem{rmk}[thm]{Remark}

\def\RR{\mathbb{R}}
\def\NN{\mathbb{N}}
\def\EE{\mathbb{E}}
\def\PP{\mathbb{P}}
\def\P{\mathcal{P}}
\def\Psym{\mathcal{P}_{\textnormal{sym}}}
\def\ind{\mathbf{1}}
\def\Id{\operatorname{Id}}
\def\ALip{\operatorname{ALip}}
\def\law{\operatorname{Law}}

\renewcommand{\mid}{\,\vert\,}

\title{Chaos for rescaled measures on Kac's sphere}

\author{Roberto Cortez\footnote{Universidad Andr\'es Bello, Departamento de Matem\'aticas, Santiago, Chile. Supported by ANID Fondecyt Iniciaci\'on Grant 11181082. E-mail: \texttt{roberto.cortez.m@unab.cl}}
\, and Hagop Tossounian\footnote{Universidad de Chile, DIM-CMM, Santiago, Chile. Supported by ANID Fondecyt Postdoctoral Grant 3200130. E-mail: \texttt{htossounian@cmm.uchile.cl}.
}
}

\begin{document}

\maketitle

\begin{abstract}
In this article we study a relatively novel way of constructing chaotic sequences of probability measures supported on Kac's sphere, which are obtained as the law of a vector of $N$ i.i.d.\ variables after it is rescaled to have unit average energy. We show that, as $N$ increases, this sequence is chaotic in the sense of Kac, with respect to the Wasserstein distance, in $L^1$, in the entropic sense, and in the Fisher information sense. For many of these results, we provide explicit rates of polynomial order in $N$. In the process, we improve a quantitative entropic chaos result of Haurey and Mischler by relaxing the finite moment requirement on the  densities from order $6$ to $4+\epsilon$.
\end{abstract}


\medskip


\section{Introduction}

\subsection{Chaotic sequences}

In this article we study a relatively novel way of constructing chaotic sequences supported on Kac's sphere. We are largely motivated by the work of Carlen, Carvalho, Le Roux, Loss, and Villani \cite{carlen-carvalho-leroux-loss-villani2010}. In this setting, ``chaos'' is to be understood as ``asymptotic statistical independence'', when a parameter $N \in \NN$, representing the number of particles, goes to infinity. More precisely:

\begin{defi}[Kac's chaos]
For each $N\in\NN$, let $F^N \in \Psym(\RR^N)$, and let $f \in \P(\RR)$. The sequence $(F^N)_{N\in\NN}$ is said to be \emph{Kac chaotic} to $f$, or simply \emph{$f$-chaotic}, if for all $k\in\NN$, the projection of $F^N$ on its first $k$ variables (any $k$ variables) converges weakly to $f^{\otimes k}$ as $N\to\infty$. That is, for any $\phi : \RR^k \to \RR$ continuous and bounded, one has
\[
\lim_{N\to \infty} \int_{\RR^N} \phi(x_1,\ldots,x_k) F^N(dx)
= \int_{\RR^k} \phi(x_1,\ldots,x_k) f(dx_1) \cdots f(dx_k).
\]
\end{defi}

Here $\P(E)$ denotes the space of probability measures on the metric space $E$, and $\Psym(\RR^N)$ is the space of symmetric probability measures on $\RR^N$, that is, those invariant under any permutation of the variables $x_1,\ldots,x_N$.

\emph{Kac's sphere} is the $(N-1)$-dimensional sphere of radius $\sqrt{N}$ embedded in $\RR^N$, that is,
\[
    S^{N-1} = \left\{ x \in \RR^N : \sum_{i=1}^N x_i^2 = N \right\}.
\]
It corresponds to the set of average energy equal to 1. Its importance comes from the fact that $S^{N-1}$ is the natural state space of \emph{Kac's $N$-particle system}: a pure-jump Markov process on $\RR^N$ representing the evolution of the one-dimensional velocities of $N$ identical particles subjected to random energy-preserving collisions; it is a simplification of the particle system associated with the spatially homogeneous Boltzmann equation for dilute gases. In his celebrated paper \cite{kac1956}, Kac proved that this model satisfies what is now known as \emph{propagation of chaos}: if the distribution of the system is chaotic at $t=0$, then it is also chaotic for later times $t>0$. This provides a bridge between the detailed microscopic description of the gas, given by the particle system, and its reduced macroscopic behaviour, given by the so-called \emph{Boltzmann-Kac equation}. We refer the reader to \cite{carlen-carvalho-leroux-loss-villani2010,mischler-mouhot2013,sznitman1989} for more information about Kac's model, the Boltzmann equation, and the important problem of propagation of chaos.

It is then natural to study chaotic sequences of distributions supported on $S^{N-1}$. An important and archetypal example is the uniform distribution on $S^{N-1}$, denoted $\sigma^N$, which is the unique equilibrium distribution of the $N$-particle system. As it is well known, $\sigma^N$ is chaotic to the Gaussian density
\[
\gamma(x) = \frac{1}{\sqrt{2\pi}} e^{-x^2/2},
\]
which in turn is the unique equilibrium of the Boltzmann-Kac equation.

When working with chaotic sequences on Kac's sphere, it is very desirable to have explicit rates of chaoticity (in $N$), and one natural way to quantify chaos is in the $L^2(d\sigma^N)$ sense. However, as explained in \cite{carlen-carvalho-leroux-loss-villani2010}, the $L^2$ norm of a chaotic sequence tends to behave badly: it can grow exponentially with $N$. A better alternative is to use \emph{entropy}: the \emph{relative entropy} of $F^N \in \P(\RR^N)$ with respect to $G^N(\RR^N)$ is given by
\[
    H(F^N \mid G^N)
    = \int_{\RR^N} h \log h \, dG^N
    = \int_{\RR^N} \log h \, dF^N
    \geq 0,
    \qquad h = \frac{dF^N}{dG^N},
\]
and $H(F^N \mid G^N) = +\infty$ when $F^N$ is not absolutely continuous with respect to $G^N$. A crucial advantage of the relative entropy over the $L^2$ norm is \emph{extensivity}: for $f, g \in \P(\RR)$ such that $H(f \mid g) < \infty$, one has
\[
H(f^{\otimes N} \mid g^{\otimes N})
= N H(f \mid g).
\]
For a sequence $F^N \in \P(S^{N-1})$ which is $f$-chaotic, one expects a similar relation to hold approximately. The precise definition, introduced in \cite[Definition 8]{carlen-carvalho-leroux-loss-villani2010}, is the following:

\begin{defi}[entropic chaos]
For each $N \in \NN$, let $F^N \in \Psym(S^{N-1})$, and let $f \in \P(\RR)$ be such that $H(f \mid \gamma) < \infty$. The sequence $(F^N)_{N\in\NN}$ is said to be \emph{entropically chaotic} to $f$ if it is Kac chaotic to $f$, and
\begin{equation}
    \label{eq:def_entropic_chaos}
    \lim_{N\to\infty} \frac{1}{N} H(F^N \mid \sigma^N)
    = H(f \mid \gamma).
\end{equation}
\end{defi}

This definition can be thought of as ``asymptotic extensivity'' of the entropy. It is a stronger notion of chaos, involving all the $N$ variables and not just a fixed number of marginals.

\subsection{Measures rescaled to Kac's sphere}

Given $f\in\P(\RR)$, consider the problem of finding a sequence $F^N$ of measures supported on Kac's sphere that is $f$-chaotic, or better, $f$-entropically chaotic. One way of obtaining such a sequence is to take the tensor product $f^{\otimes N}$ and \emph{condition} (restrict) it to Kac's sphere. The idea is that, if one assumes $\int_\RR x^2 f(dx) = 1$, then, by the law of large numbers, one has $\sum_i x_i^2 \approx N$ under the law $f^{\otimes N}$; this means that $f^{\otimes N}$ is already concentrated around $S^{N-1}$, and the conditioning should not change it too much. For a bounded density with finite fourth moment, it is proven in \cite[Theorem 9]{carlen-carvalho-leroux-loss-villani2010} that the sequence obtained with this construction is indeed entropically chaotic to $f$. See also \cite{carrapatoso2015} for similar results in the context of \emph{Boltzmann's sphere}.

In the present article, we propose the following alternative way of constructing a chaotic sequence on Kac's sphere, based on \emph{rescaling} instead of conditioning:

\begin{defi}[rescaled measure]
For $x \in \RR^N \backslash \{0\}$, we denote $\hat{x} \in S^{N-1}$ the vector $x$ \emph{rescaled} to Kac's sphere, that is,
\[
    \hat{x}
    = \frac{\sqrt{N}}{\vert x\vert } x.
\]
Similarly, for $F^N \in \P(\RR^N)$ without an atom at the origin, we define the \emph{rescaled} probability measure $\hat{F}^N \in \P(S^{N-1})$ as the push forward of $F^N$ by the mapping $x \in \RR^N \backslash \{0\} \mapsto \hat{x} \in S^{N-1}$, that is,
\[
\int_{S^{N-1}} \phi(x) \hat{F}^N(dx)
= \int_{\RR^N} \phi(\hat{x}) F^N(dx).
\]
Moreover, given $f\in\P(\RR)$ without an atom at 0, we denote $\hat{f}^N = \hat{F}^N$ for $F^N = f^{\otimes N}$.
\end{defi}

For example, it can be easily seen that $\hat{\gamma}^N = \sigma^N$, by rotational symmetry of $\gamma^{\otimes N}$. This will be used several times throughout this article. Moreover, if $Z \sim \gamma^{\otimes N}$, then $\vert Z\vert $ and $\hat{Z} \sim \sigma^N$ are independent.

We remark that, to the best of our knowledge, the earliest use of this kind of rescaling in the setting of kinetic theory can be found in the proof of \cite[Lemma 25]{fournier-guillin2017}, where the authors provide explicit rates of chaoticity for the uniform distribution on Boltzmann's sphere towards the Gaussian on $\RR^3$. The first results for general rescaled measures seems to be found in \cite[Section 5]{cortez-fontbona2018}; see also \cite{cortez2016,cortez-tossounian2021AIHP}. In all of these references, some chaos estimates are proven using the 2-Wasserstein distance. The main goal of the present article is to investigate other notions of chaos satisfied by sequences of rescaled measures, in particular entropic chaos.

The main feature of the rescaled measure is its \emph{simplicity}. For instance, its definition is straightforward and is valid for any $F^N \in \P(\RR^N)$ without an atom at the origin\footnote{If $f\in\P(\RR)$ has an atom at 0 (but it's not the Dirac mass at 0), one can still define $\hat{f}^N$ in such a way that it is $f$-chaotic, see \cite{cortez-tossounian2021AIHP} for details.}, while the conditioning procedure is typically applied only to tensor product measures $f^{\otimes N}$ and requires some smoothness and integrability assumptions on $f$, see \cite[Definition 7]{carlen-carvalho-leroux-loss-villani2010}. Similarly, as we shall see, the results one can obtain for rescaled measures require less assumptions, and the proofs tend to be simpler. For instance, some of our proofs make use of classical limit theorems such as the Law of Large Numbers and the Central Limit Theorem, while the analysis of conditioned measures requires a refined local version of the latter (see \cite[Appendix A]{carlen-carvalho-leroux-loss-villani2010}). Another nice feature is that it is straightforward to sample from $\hat{f}^N$: simply generate $X \sim f^{\otimes N}$ and then compute $\hat{X} = \sqrt{N} X / \vert X\vert  \sim \hat{f}^N$; this is an efficient way to generate a suitable initial condition when simulating Kac's particle system.

\subsection{Main contribution and structure of the article}

In this work we prove several results concerning chaoticity and related properties for sequences of measures constructed via rescaling to Kac's sphere. We split our presentation as follows.

In Section \ref{sec:Kacs_chaos} we study Kac's chaos. In Theorem \ref{thm:chaos_rescaled_measures} we prove that the rescaled measure of any $f$-chaotic sequence (not necessarily a tensor product) is also $f$-chaotic, under a uniform $p>4$ moment assumption. In Lemma \ref{lem:moments} we prove that the rescaled measure of a tensor product has exactly the same uniform (in $N$) moments as the the finite moments of the reference measure. In Theorem \ref{thm:chaos_rescaled_tensor} and Corollary \ref{cor:chaos_Wr} we prove that the rescaling of a tensor product is always chaotic, with some estimates in the Wasserstein metric.

Section \ref{sec:L1-chaos} considers chaos in the sense of (strong) $L^1$. We show in Theorem \ref{thm:POC_L1} that if $f$ has $2+\epsilon$ moments, and belongs to a broad subspace $\ALip(r)$ of $L^1$, which we introduce in Definition \ref{defi:ALip}, then the $k$-marginals $\Pi_k \hat{f}^N$ converge strongly in $L^1$ to $f^{\otimes k}$. The convergence is of order $N^{-\eta}$, with $\eta$ explicitly given in \eqref{eq:eta}. This section also provides some examples of functions in $\ALip(r)$ in Lemma \ref{lem:exponents} and a property of the space $\ALip(r)$ in Lemma \ref{lem:Besov}. This section is independent of the rest.

Section \ref{sec:entropic_chaos} is devoted to entropic chaos. One of our main results is Theorem \ref{thm:entropic_chaos_rescaled_tensor}, where we prove that the rescaled tensor product of $f$ is always $f$-entropically chaotic, requiring only that $f$ have unit energy and finite entropy relative to $\gamma$. Perhaps surprisingly, the proof is quite straightforward: one can easily see that $H(\hat{f}^N \mid \sigma^N)$ is always smaller than $N H(f \mid \gamma)$ by writing their difference as the negative of some relative entropy, see Lemma \ref{lem:HhatF-HF}. The other inequality in \eqref{eq:def_entropic_chaos} follows from \cite[Theorem 12]{carlen-carvalho-leroux-loss-villani2010}, which we recall in Theorem \ref{thm:semi-extens}; we provide a simpler proof in the Appendix. In Theorem \ref{thm:quantitative_entropic_chaos_for_rescaled_tensor_products}, we provide a quantitative rate of entropic chaos of mild polynomial order in $N$ for the rescaled tensor product. The proof relies on a control on the \emph{Fisher information} of the sequence, provided in the next section, and on an improved version of \cite[Theorem 4.17]{hauray-mischler2014}, given in Theorem \ref{thm:H&M_JFA2014}, which relaxes the finite $6$-moments condition on $f$ to just $4+\epsilon$. While proving this, we also improved \cite[Theorem 4.13]{hauray-mischler2014}, which is a quantitative entropic chaos result concerning \emph{conditioned tensor product measures} (defined in Definition \ref{def:conditioned-states}), again by relaxing the moments requirements. This is done in Lemma \ref{lem:conditioned-product-improvement}. 

In Section \ref{sec:fisher_information_chaos} we study the even stronger notion of \emph{Fisher information chaos} (see Definition \ref{def:fisher_information}). We show in Theorem \ref{thm:fisher_chaos} that the rescaled tensor product measures $(\hat{f}^N)_N$ are Fisher-information chaotic to $f$ under mild assumptions. The proof requires several computational results given by lemmas and propositions \ref{prop:fisher-main}-\ref{lem:limit_of_x_f_x}.

In Section \ref{sec:conclusion} we give a general conclusion and mention some open problems. In the Appendix we provide proofs for some of the technical results.

\subsection{Notation}

Let us fix some notation:
\begin{itemize}
    \item $\P(E)$ denotes the space of probability measures on the metric space $E$; $\P_p(E)$ is the subspace of probability measures with finite $p$ moment. $\Psym(\RR^N)$ denotes the space of symmetric probability measures on $\RR^N$, i.e., those invariant under permutations of the $N$ coordinates. $C_b(E)$ denotes the space of continuous and bounded real functions on $E$.
    
    \item Given $F^N \in \P(\RR^N)$, we denote $\Pi_k F^N \in \P(\RR^k)$ the projection of $F^N$ on its first $k \leq N$ variables.
    
    \item $S^{N-1}(r) = \{x \in \RR^N : \vert x\vert ^2 = r \}$ is the $(N-1)$-dimensional sphere of radius $r$ on $\RR^N$, where $\vert x\vert =(x_1^2 + \cdots x_N^2)^{1/2}$ denotes the usual norm. Thus, Kac's sphere is $S^{N-1} = S^{N-1}(\sqrt{N})$.
    
    \item If $F^N \in\P(\RR^N)$ has a density with respect to the Lebesgue measure, we will abuse notation and denote $F^N(x)$ its density.
    
    \item $W_p$ denotes the $p$-Wasserstein distance on $\P_p(\RR^N)$, that is, for $F^N,G^N\in \P_p(\RR^N)$,
    \[
    W_p(F^N,G^N)
    = \left( \inf_\pi \int_{\RR^N \times \RR^N} \vert x-y\vert ^p \pi(dx,dy) \right)^{1/p},
    \]
    where the infimum is taken over all $\pi \in \P(\RR^N \times \RR^N)$ with $F^N$ and $G^N$ as first and second marginals, respectively.
    
    \item $\nabla_S$ denotes the \emph{spherical gradient} on $S^{N-1}$, i.e., for any $f \in C^1(S^{N-1})$ and $y \in S^{N-1}$,
    \[
        \nabla_S f (y) = \nabla \tilde{f}(y) - \frac{1}{N}  \nabla \tilde{f}(y) \cdot y,
    \]
    where $\tilde{f}$ is any $C^1$ extension of $f$ to a neighbourhood of $y$ in $\RR^N$.
    
    \item The \emph{relative Fisher information} of $F^N \in \P(\RR^N)$ with respect to $G^N \in \P(\RR^N)$ is given by
    \[
    I(F^N \mid G^N)
    = \int_{\RR^N} \frac{\vert \nabla h(x) \vert^2}{h(x)} G^N(dx),
    \qquad h=\frac{dF^N}{dG^N},
    \]
    and $ I(F^N \mid G^N) = +\infty$ when $F^N$ does not have a density with respect to $G^N$ with continuous first derivatives. The relative Fisher information of $F^N \in \P(S^{N-1})$ with respect to $G^N \in \P(S^{N-1})$ is defined similarly, but replacing the usual gradient $\nabla$ by the spherical gradient $\nabla_S$.
\end{itemize}

\section{Kac's chaos}
\label{sec:Kacs_chaos}

We start by stating a law of large numbers in $L^2$ for chaotic sequences.

\begin{lem}[law of large numbers for chaotic sequences]
\label{lem:LLN_chaotic}
Let $F^N \in \P(\RR^N)$ be an $f$-chaotic sequence, for some $f \in \P(\RR)$, such that $\sup_N \int_{\RR^N} \vert x_1\vert ^p F^N(dx) < \infty$ for some $p>2$. Let $X = (X_1,\ldots,X_N)$ be $F^N$-distributed. Then, for any continuous $\phi : \RR \to \RR$ with linear growth, we have
\[
\lim_{N\to \infty} \frac{1}{N} \sum_{i=1}^N \phi(X_i)
= \int_\RR \phi(x) f(dx),
\qquad \text{in $L^2$.}
\]
\end{lem}

\begin{proof}
For any such $\phi$, it is clear that $\phi(X_1)$ and $\phi(X_1) \phi(X_2)$ are uniformly (in $N$) integrable: by the uniform bound of the $p$-th moment and the linear growth assumption of $\phi$, it is easily seen that $\sup_N \EE[\phi(X_1)^p] < \infty$ and $\sup_N \EE[(\phi(X_1) \phi(X_2))^{p/2}] < \infty$. Consequently, using $f$-chaoticity,  we get
\begin{equation}
\label{eq:lem_LLN_chaotic_Ephi}
    \lim_{N\to\infty} \EE[\phi(X_1)] = m,
    \quad \text{and} \quad
    \lim_{N\to\infty} \EE[\phi(X_1)\phi(X_2)] = m^2,
\end{equation}
where $m = \int_\RR \phi(x) f(dx)$. Using exchangeability, we have:
\begin{align*}
    & \EE\left[ \left( \frac{1}{N} \sum_{i=1}^N \phi(X_i) - m \right)^2 \right] \\
    &= \EE\left[  \frac{1}{N^2} \sum_{i=1}^N \phi(X_i)^2 + \frac{1}{N^2} \sum_{i\neq j} \phi(X_i)\phi(X_j) - \frac{2m}{N} \sum_{i=1}^N \phi(X_i) + m^2 \right] \\
    &= \frac{\EE[\phi(X_1)^2]}{N} + \frac{N-1}{N}\EE[\phi(X_1)\phi(X_2)]
    - 2 m \EE[\phi(X_1)]
    + m^2,
\end{align*}
which converges to 0 as $N\to\infty$ thanks to \eqref{eq:lem_LLN_chaotic_Ephi}.
\end{proof}

The next theorem shows that the rescaled measures of a chaotic sequence are also chaotic, under a mild assumption of bounded $p$-moments for some $p>4$. The proof is straightforward, and relies on the law of large numbers stated in Lemma \ref{lem:LLN_chaotic}. This showcases the simplicity of the rescaled measures.

\begin{thm}[chaos for rescaled measures]
\label{thm:chaos_rescaled_measures}
Let $(F^N)_{N\in\NN}$ be an $f$-chaotic sequence, for some $f \in \P(\RR)$ satisfying $\int_\RR x^2 f(dx) = 1$. Assume that $F^N$ does not have an atom at the origin, and that $\sup_N \int_{\RR^N} \vert x_1\vert ^p F^N(dx) < \infty$ for some $p>4$. Then $\hat{F}^N$ is $f$-chaotic.
\end{thm}

\begin{proof}
Let $X = (X_1,\ldots,X_N)$ be $F^N$-distributed. We need to show that for any $\phi:\RR^k \to \RR$ $L$-Lipschitz continuous and bounded one has
\begin{equation}
\label{eq:proof_thm_chaos_rescaled_measures_Ephi}
    \lim_{N\to\infty} \EE[ \phi(\hat{X}_1,\ldots,\hat{X}_k) ]
    = \int_{\RR^2} \phi(x_1,\ldots,x_k) f(dx_1) \cdots f(dx_k).
\end{equation}
Without loss of generality, we assume that $E_N := \int_{\RR^N} x_1^2 F^N(dx) = 1$ for all $N$ (if not, apply the following argument to $X / \sqrt{E_N}$). Call $Q_N = \frac{1}{N}\sum_i X_i^2$ and $Y = (X_1,\ldots,X_k)$, thus $\hat{X}_i = Y_i/\sqrt{Q_N}$. Since $(F^N)$ is $f$-chaotic, we know that $\EE[\phi(Y)]$ converges to the r.h.s.\ of \eqref{eq:proof_thm_chaos_rescaled_measures_Ephi}; thus, it suffices to show that the following vanishes as $N\to\infty$:
\begin{align}
    \notag
    \left\vert  \EE[\phi(Q_N^{-1/2} Y )] - \EE[\phi(Y)] \right\vert 
    &\leq L \EE\left[ \left\vert  (Q_N^{-1/2}-1) Y \right\vert  \right] \\
    \notag
    &\leq kL \EE\left[ \left\vert  Q_N^{-1/2}-1 \right\vert  \frac{1}{N}\sum_{i=1}^N \vert X_i\vert  \right] \\
    \notag
    &\leq kL \EE\left[ \left\vert  Q_N^{-1/2}-1 \right\vert  Q_N^{1/2} \right] \\
    \notag
    &= kL \EE\left[ \left\vert  Q_N^{1/2} - 1 \right\vert  \right] \\
    \label{eq:proof_thm_chaos_rescaled_measures_EQN}
    &\leq kL \EE\left[ \left\vert  Q_N - 1 \right\vert  \right],
\end{align}
where we have used exchangeability. Finally, thanks to the uniform (in $N$) bound on the moment of $F^N$ of order $p>4$, we can apply Lemma \ref{lem:LLN_chaotic} to the sequence $(\law(X_1^2,\ldots,X_N^2))_{N\in\NN}$ with $\phi(x) = x$, and deduce that $\lim_N Q_N = 1$ in $L^2$; thus, \eqref{eq:proof_thm_chaos_rescaled_measures_EQN} converges to $0$ as $N\to\infty$.
\end{proof}

We now turn our attention to rescaled tensor products $\hat{f}^N := \hat{F}^N$ with $F^N = f^{\otimes N}$, for some $f\in \P(\RR)$. The next lemma shows that $\hat{f}^N$ has the same uniformly (in $N$) bounded one-particle moments as the finite moments of the original measure $f$:


\begin{lem}[moments for rescaled tensor products]
\label{lem:moments}
Let $f\in\P_2(\RR)$, without an atom at 0. Then, for any $p\geq 2$,
\[
\sup_N \int_{\RR^N} \vert x_1\vert ^p \hat{f}^N(dx) < \infty
\qquad \text{if and only if} \qquad
\int_{\RR} \vert x\vert ^p f(dx) < \infty. 
\]
\end{lem}

\begin{proof}
We first prove the direct implication. Let $X = (X_1,\ldots,X_N)$ be a collection of i.i.d.\ and $f$-distributed random variables, thus $\hat{X} \sim \hat{f}^N$. Call $Q_N := \frac{1}{N}\sum_i X_i^2$, thus $\hat{X} = X / \sqrt{Q_N}$. Denoting $E := \int_\RR x^2 f(dx)$, we have:
\begin{align}
\notag
\EE[ \vert \hat{X}_N\vert ^p ]
&= \EE\left[ \vert X_N\vert ^p \left( \frac{X_N^2}{N} + \frac{N-1}{N}Q_{N-1} \right)^{-p/2} \right] \\
\label{eq:proof_lem_moments_EXp}
&= \left( \frac{N}{N-1} \right)^{p/2} \int_\RR \vert x\vert ^p \EE\left[ \left( \frac{x^2}{N-1} + Q_{N-1} \right)^{-p/2} \right] f(dx) \\
\notag
&\geq \left( \frac{N}{N-1} \right)^{p/2} \int_\RR \vert x\vert ^p \left( \frac{x^2}{N-1} + E \right)^{-p/2} f(dx),
\end{align}
where in the last step we used Jensen's inequality. By the monotone convergence theorem, we deduce that
\[
\liminf_{N\to\infty} \EE[\vert \hat{X}_N\vert ^p]
\geq E^{-p/2} \int_\RR \vert x\vert ^p f(dx),
\]
which proves the direct implication.

Now we prove the converse: given any $A>0$, we split the expectation in \eqref{eq:proof_lem_moments_EXp} in the cases $Q_{N-1}\geq A$ and $Q_{N-1}<A$, which gives
\begin{align*}
\EE[ \vert \hat{X}_N\vert ^p ]
&\leq \left( \frac{N}{N-1} \right)^{p/2} \int_\RR \vert x\vert ^p \left\{  A^{-p/2} + \left( \frac{x^2}{N-1} \right)^{-p/2} \PP(Q_{N-1}<A) \right\} f(dx) \\
&= \left( \frac{N}{N-1} \right)^{p/2} A^{-p/2} \int_\RR \vert x\vert ^p f(dx)
+ N^{p/2} \PP(Q_{N-1} < A).
\end{align*}
Thus, it suffices to show that for some $A>0$, the second term in the last expression is bounded uniformly in $N$. To this end, for each $i=1,\ldots,N-1$, consider the random variable $Y_i := \ind_{\{X_i^2 \geq E\}}$, thus $Y_1,\ldots,Y_{N-1}$ are i.i.d.\  $\textnormal{Bernoulli}(q)$ for $q=\PP(X_1^2 \geq E) > 0$. Clearly $E Y_i \leq X_i^2$, thus for $A = Eq/2$ we have
\[
\PP(Q_{N-1} < A)
\leq \PP\left(\frac{1}{N-1} \sum_{i=1}^{N-1} Y_i < \frac{q}{2} \right)
\leq e^{-q(N-1)/8},
\]
where we have used the Chernoff-type bound $\PP(\sum_{i=1}^n Y_i \leq n (1-\delta) q) \leq e^{-\delta^2 n q/2}$ (valid for any $0<\delta<1$), with $\delta = 1/2$. We thus deduce that $ \sup_N N^{p/2} \PP(Q_{N-1} < A) < \infty$ for $A=Eq/2$, as desired. This concludes the proof.
\end{proof}


The next theorem shows that $\hat{f}^N$ is always $f$-chaotic, under no assumption other than unit energy and no atom at 0. The proof is straightforward, and relies on the usual law of large numbers. Moreover, under the additional assumption that $f$ has finite $p$-moment for some $p>2$, using \cite[Theorem 3]{cortez-tossounian2021AIHP}, we can provide an explicit rate of chaoticity of polynomial order, in the $2$-Wasserstein distance. We remark that this is a ``strong'' chaos result (see also Corollary \ref{cor:chaos_Wr}), in the sense that it compares $\hat{f}^N$ and $f^{\otimes N}$ directly, and not just a fixed number of marginals.

\begin{thm}[chaos for rescaled tensor products]
\label{thm:chaos_rescaled_tensor}
Let $f\in\P(\RR)$, without an atom at 0, and such that $\int_\RR x^2 f(dx) = 1$. Then $\hat{f}^N$ is chaotic to $f$. Moreover, if $f\in\P_p(\RR)$ for some $p>2$, then we have the following quantitative  rate: there exists a constant $C$ depending only on $p$ and $\int_\RR \vert x\vert ^p f(dx)$, such that
\begin{equation}
\label{eq:thm_chaos_for_rescaled_tensor_W2}
\frac{1}{N} W_2(\hat{f}^N, f^{\otimes N})^2 \leq C
\begin{cases}
    N^{-1/2}, & p>4, \\
    N^{-1/2}\log(N), & p=4, \\
    N^{-(1-2/p)}, &  p \in (2,4).
\end{cases}
\end{equation}
\end{thm}

\begin{proof}
Using the same argument and notation as in the proof of Theorem \ref{thm:chaos_rescaled_measures}, the $f$-chaoticity of $\hat{f}^N$, assuming only that $\int_\RR x^2 f(dx) = 1$, follows from \eqref{eq:proof_thm_chaos_rescaled_measures_EQN} and the usual law of large numbers in $L^1$ for i.i.d.\ sequences. The proof of \eqref{eq:thm_chaos_for_rescaled_tensor_W2} can be found in \cite[Theorem 3]{cortez-tossounian2021AIHP}, so we omit it here.
\end{proof}

As a corollary, we can bound $W_r(\hat{f}^N, f^{\otimes N})^r$ for $r>2$ by imposing finite $p>r$ moments on $f$ (see \cite[Equation $2.11$] {hauray-mischler2014} for a special case).

\begin{cor}
\label{cor:chaos_Wr}
Let $f\in \P_p(\RR)$ for some $p>2$, without an atom at 0. Then, for any $2< r < p$, there exists a constant $C>0$ such that, for $b = \frac{p-r}{p-2} < 1$, we have 
\[
\frac{1}{N} W_r(\hat{f}^N, f^{\otimes N})^r \leq C
\begin{cases}
    N^{-b/2}, & p>4, \\
    N^{-b/2} \log(N)^b, & p=4, \\
    N^{-b(1-2/p)}, &  p \in (2,4).
\end{cases}
\]
\end{cor}

\begin{proof} Let $U$ and $V$ be random vectors on $\RR^N$ such that $U\sim f^{\otimes N}$, $V \sim \hat{f}^N$, and $W_2(f^{\otimes N},\hat{f}^N)^2 = \EE \vert U - V \vert^2$. Moreover, our hypotheses, together with Lemma \ref{lem:moments}, imply that $M_p := \EE \vert U_1 \vert^p + \EE\vert V_1 \vert^p < \infty$. Noting that $p = \frac{r-2b}{1-b}$ and using H\"{o}lder's inequality with exponents $\frac{1}{1-b}$ and $\frac{1}{b}$, we have:
\begin{align*}
    \frac{1}{N} \EE \vert U - V \vert^r
    &= \frac{1}{N} \EE \vert U - V \vert^{r-2b} \vert U - V \vert^{2b} \\
    &\leq \frac{1}{N} \left(\EE \vert U - V \vert^p \right)^{1-b}
    \left(\EE \vert U - V \vert^2 \right)^b \\
    &= \left( \frac{1}{N} \EE \vert U - V \vert^p \right)^{1-b}
    \left( \frac{1}{N} \EE \vert U - V \vert^2 \right)^b \\
    &\leq (2^{p-1} M_p)^{1-b}  \left( \frac{1}{N} W_2(f^{\otimes N},\hat{f}^N)^2 \right)^b,
\end{align*}
where we used exchangeability and the inequality $\vert U_i - V_i \vert^p \leq 2^{p-1}(\vert U_i \vert^p + \vert V_i \vert^p)$. Since $\EE \vert U - V \vert^r$ is an upper estimate of $W_r(f^{\otimes N},\hat{f}^N)^r$, the conclusion follows from Theorem \ref{thm:chaos_rescaled_tensor}.
\end{proof}

\section{Chaos in \texorpdfstring{$L^1$}{L1}}
\label{sec:L1-chaos}

The aim of this section is to prove Theorem \ref{thm:POC_L1} which provides a propagation of chaos estimate in $L^1$ for the family $\{ \hat{f}^N \}_N$, whenever $f$ belongs to a subspace $\ALip(r)$ of $L^1$ which we introduce in definition \ref{defi:ALip}. This space contains all piecewise Lipschitz functions that have at most countably many jumps and countably many blowups of the type $\vert x-x_0\vert^{-a}$ with $a \leq a_0$, for some fixed $a_0=a_0(r)< 1$.

\begin{defi}[The almost Lipschitz spaces] \label{defi:ALip} Let $r\geq 0$ and $L_0<\infty$. We first define the set $\ALip(r,L_0)$ as follows
\[
\begin{split}
    \ALip(r, L_0) = \left\{ f \in L^1(\RR):  \text{s.t.} \forall \epsilon>0, \exists g_\epsilon \in L^1 \text{ and Lipschitz }, \right. \\ \text{ such that }  \left. \Vert f-g_\epsilon\Vert_{L^1}\leq \epsilon \text{ and } \Vert g_\epsilon\Vert_{\text{Lip}} \leq L_0 \epsilon^{-r}\right\}.
\end{split}
\]
The almost-Lipschitz space is given by
\[
    \ALip(r) = \cup_{0<L_0<\infty} \ALip(r, L_0)
\]
\end{defi}
We will provide some properties of functions in $\ALip(r)$ after the proof of Theorem \ref{thm:POC_L1} in Lemmas \ref{lem:almost-Lipschitz} and \ref{lem:Besov}.

\begin{thm}\label{thm:POC_L1} Fix $r\geq 0$ and let $f$ be a probability density on $\RR$ that belongs to $\ALip(r)$. Assume that $\int_\RR x^2 f(x) dx = 1$ and that $\int_\RR \vert x\vert^{2+\delta} f(x) dx< \infty$ for some $\delta \in (0,2]$. Then there exists a constant $C_0$ independent of $N$ such that, for all $k \leq N$,
\[
 \Vert \Pi_k \hat{f}^N - f^{\otimes k} \Vert_{L^1(\RR^k)}  \leq C_0 N^{-\eta},
\]
where $\eta$ is given by
\begin{equation}\label{eq:eta}
\eta = \frac{\delta}{k + 5 + \delta + r(2+\delta)}.
\end{equation}
\end{thm}

\begin{proof}
Without loss of generality, we will take $N$ to be at least $N_0$ given by
\begin{equation}
  N_0(k, \delta) = \left(2 k\right)^{1+\delta/2} \label{eq:Nmin}.
  \end{equation}

For brevity, we will let
\begin{itemize}
\item $T_R= \int_{B(0,R)^c} f^{\otimes k}(x) d^k x$
\item $X_1, X_2, \dots$  i.i.d. with law $f$, $\mu_m = $(law of $X_1^2+\dots + X_m^2$).
\item $\tilde{S}_{N,k}= \left\{ \vert X_1^2+ \dots + X_{N-k}^2- N  \vert >N^{1-q}\right\}$, where 
\item $q \in (0,1)$ is a fixed number which will be optimized at the end of the proof
\end{itemize}

In order to study the the $k$-marginal of $\hat{f}^N$, we will check its action on $\phi \in L^\infty(\RR^k)$, namely
\[
	L_N[\phi]:=\int_{S^{N-1}} \phi(x_1,\dots, x_k) \frac{d \hat{f}^N}{d\sigma^N}(x)\sigma^N(dx).
\]
We will make a sequence of transformations on $L_N[\phi]$ and remove from it terms which become small in $L^1$ as $N\to \infty$ in a quantitative way. We have,
\begin{eqnarray*}
L_N[\phi] &= & \int_{\RR^N \backslash \{0\}} f^{\otimes N}(x) \phi\left(\frac{\sqrt{N}}{\vert x\vert } (x_1, \dots, x_k) \right) d^N x \\
	& =&  \int_{(0,\infty)} \mu_{N-k}(ds) \left( \int_{\RR^k} f^{\otimes k}(y) \phi\left(\sqrt{\frac{N}{\vert  y \vert ^2+s }} y \right) d^k y \right)\\
	& = &  E_1[\phi] + \int_{\vert s- N\vert\leq N^{1-q}} \mu_{N-k}(ds) \left( \int_{\RR^k} f^{\otimes k}(y) \phi\left(\sqrt{\frac{N}{\vert  y \vert ^2+s }} y \right) d^k y \right)
\end{eqnarray*}
where $\vert E_1[\phi] \vert \leq \Vert \phi \Vert_{L^\infty} \PP(\tilde{S}_{N,k})$, with $\tilde{S}_{N,k}$ defined at the beginning of the proof.
Next, we cutoff $\vert  y \vert ^2$ at $(N-k)^{(1-q)/2}$. So that 
\begin{equation}\label{eq:thm1_proof_helpful}
    L_N[\phi] =  E_1[\phi] + E_2[\phi] +  \int_{\{\vert s- N\vert\leq N^{1-q}\} }  \mu_{N-k}(ds) \left( \int_{\RR^k} f^{\otimes k}(y) \phi\left( \psi_{N,s}(y) \right) d^k y \right),
\end{equation}
where $\psi_{N,s}$ is defined by
\[
    \psi_{N,s} (y) = \sqrt{\frac{N}{\min \{\vert  y \vert ^2, (N-k)^{(1-q)/2}\}+s }} y,
\]
and $\vert E_2[\phi]\vert \leq 2\Vert \phi\Vert_{L^\infty} T_{ N^{(1-q)/4}}$. We will analyze all the terms in \eqref{eq:thm1_proof_helpful} to show that under our assumptions on $f$, $L_N[\phi]$  approximates $\int_{\RR^k} f^{\otimes k}(y) \phi(y) d^k y$ well.

We first find the rates at which $\PP(\tilde{S}_{N,k})$ (and hence $E_1[\phi]$) approaches zero. A quantitative law of large numbers from \cite{vonBahr-esseen1965}, reproduced in Lemma \ref{lem:vonBahr} below, implies that (details in Corollary \ref{cor:vonBahr})
\[
    \PP(\tilde{S}_{N,k}) \leq  2^{2+\delta/2} N^{-(\delta/2-(1+\delta/2)q)} \EE \vert X_1\vert^{2+\delta},
\]
making 
\[
    \vert E_1[\phi]\vert \leq 8 \Vert \Phi\Vert_\infty \EE\vert X_1\vert^{2+\delta} N^{-(\delta/2-(1+\delta/2)q)}.
\]
Next, we use the finite $(2+\delta)$-moment assumption on $f$ to obtain a bound for the term $T_{N^{(1-q)/4}}$. We observe that whenever $x_1^2 + \dots + x_k^2 \geq R^2$, we have $ \vert x_1\vert^{2+\delta} + \dots + \vert x_k\vert^{2+\delta} \geq k \frac{R^{2+\delta}}{k^{1+\delta/2}}$. Thus, we have:
\[
	T_R \leq \int_{\RR^k} \frac{k^{1+\delta/2}}{k R^{2+\delta}}\sum_{j=1}^k  \vert x_j\vert^{2+\delta} f^{\otimes k}(x) d^k x
	=\frac{k^{1+\delta/2}}{ R^{2+\delta}} \EE \vert X_1\vert^{2+\delta},
\]
which implies the following.
\[
    \vert E_2[\phi] \vert \leq 2\Vert \phi\Vert_{L^\infty} k^{1+\delta/2} \EE \vert X_1\vert^{2+\delta}  N^{-(1-q)(2+\delta)/4}.
\]
Third, we show that the $y$-integral in \eqref{eq:thm1_proof_helpful} approximately equals $\int_{\RR^k} f^{\otimes k}(y) \phi(y)\,d^k y$ uniformly over the range allowed for $s$. In Lemmas \ref{lem:psi-properties} and \ref{lem:quantitative}, we show that 
$\psi_{N,s}^{-1} (y) \approx y$ in the sense that 
\[
    \vert \psi_{N,s}^{-1}(y) - y \vert \leq \epsilon_N \vert y\vert, \, \Vert D\psi_{N,s}^{-1}(y)- \Id \Vert \leq \epsilon_N
\]
with $\epsilon_N \leq C(k,q) N^{-q}$. Thus, we can use Lemma \ref{lem:L1-comparison-non-Lipschitz} below with $\beta= 2+\delta$ to obtain
\[
    \Vert f^{\otimes k}(x) - f^{\otimes k}(\psi_{N,s}^{-1}(x)) \vert \det  D\psi_{N,s}^{-1}(x) \vert \Vert_{L^1(\RR^k)} \leq C N^{-q \frac{2+\delta}{k+3+\delta}}.
\]
which holds uniformly over the range $s\in [N-N^{1-q}, N+N^{1+q}]$
provided that $ N\geq (2k)^{1/(1-q)}$. Here we have substituted $\epsilon_N = C N^{-q}$.

Fourth, we consider the fact that in the $s$ integral in \eqref{eq:thm1_proof_helpful}, $\mu_{N-k}([N-N^{1-q}, N+N^{1+q}]) < 1$. Thus, we can instead compare $L_N[\phi]$ to $\mu_{N-k}([N-N^{1-q}, N+N^{1+q}]) \int_{\RR^k} f^{\otimes k}(x) \phi(x) d^k x$. This comes at a cost which is comparable to our bound on $\vert E_1[\phi]\vert$.

Combining all the parts above leads us to the fact that for any $q \in (0, 1)$, we have the bound
\begin{eqnarray*}
	 \Vert \Pi_k \hat{f}^N - f^{\otimes k} \Vert_{L^1(\RR^k)} & \leq &  C(k, \delta, L_0, \EE\vert X_1\vert^{2+\delta} ) N^{-q (2+\delta)/(k+3+\delta)} + \\
		&	&  C\left(\delta, \EE\vert X_1\vert^{2+\delta} \right) N^{-(\delta/2- q(1+\delta/2))} + \\
		&   & C\left(k, \EE\vert X_1\vert^{2+\delta} \right) N^{-(1-q)(2+\delta)/4},
\end{eqnarray*}
whenever $N\geq \max\{ 2k , (2 k \EE X_1^2)^{1/(1-q)}\}$.
It remains to choose $q$ in order to maximize the slowest decay rate. This is done in Lemma \ref{lem:exponents}, which implies that
\[
 \Vert \Pi_k \hat{f}^N - f^{\otimes k} \Vert_{L^1}  \leq  C\left(k, \delta, L_0,\EE\vert X_1\vert^{2+\delta} \right) N^{-\delta/(k+5+\delta+r(2+\delta))}.
\]
Here, the exponent $1+\delta/2$ in \eqref{eq:Nmin} comes from a simple upper bound to $1/(1-q)$ for the above optimal $q$.
\end{proof}

\begin{lem}\label{lem:psi-properties} Let $a, b$, and $c$ be positive numbers, and consider the transformation $\psi: \RR^k \to \RR^k$ given by
\begin{equation}\label{eq:psi}
	z= \psi(x)=\left( \frac{a}{b+\min\{\vert x\vert^2,c\}} \right)^{1/2} x.
\end{equation}
Then this transformation is one-to-one and onto; its inverse is given by
\begin{equation}\label{eq:psi_inverse}
	\psi^{-1}(x)= \left\{ \begin{array}{cc} \sqrt{\frac{b}{a-\vert z \vert^2} }z, &  \vert z \vert \leq \sqrt{\frac{ac}{b+c}} \\ \sqrt{\frac{b+c}{a}}z, & \vert z \vert > \sqrt{\frac{ac}{b+c}}\end{array}\right..
\end{equation}
Also $\psi^{-1}$ has a distributional derivative given by
\begin{equation}\label{eq:Dpsi_inverse}
D\psi^{-1}(z) =  \left\{ \begin{array}{cc} \sqrt{\frac{b}{a-\vert z \vert^2} }\left( \Id - \frac{z z^T}{a-\vert z \vert^2 }\right), &  \vert z \vert < \sqrt{\frac{ac}{b+c}} \\ \sqrt{\frac{b+c}{a}} \Id, & \vert z \vert > \sqrt{\frac{ac}{b+c}}\end{array}\right. .
\end{equation}
Finally, the Jacobian of $\psi^{-1}$ is given as follows
\[
\vert \det D\psi^{-1}(z)\vert =  \left\{ \begin{array}{cc} \left(\frac{b}{a-\vert z \vert^2} \right)^{\frac{k}{2}}\left( 1 - \frac{\vert z \vert^2}{a-\vert z \vert^2 }\right), &  \vert z \vert < \sqrt{\frac{ac}{b+c}} \\ \left(\frac{b+c}{a}\right)^{\frac{k}{2}}, & \vert z \vert > \sqrt{\frac{ac}{b+c}}\end{array}\right.
\]
\end{lem}

\begin{proof}
The proof of \eqref{eq:psi}-\eqref{eq:psi_inverse}-\eqref{eq:Dpsi_inverse}
follows from a direct computation. In order to compute the Jacobian, one can use the identity $\det\left(\Id - a_0 v v^T\right) = 1- a_0 \vert v\vert^2$ for any $a_0 \in \RR$ and $v \in \RR^k$.
\end{proof}

In the following lemma, we find an explicit number $\epsilon_N$ for which $\vert \psi^{-1}(z) - z\vert \leq \epsilon_N \vert z\vert $ and $\Vert D\psi^{-1}(z) - \Id \Vert \leq \epsilon_N$, and $\vert 1 - \vert \det D\psi^{-1}(x) \vert \vert \leq \epsilon_N$ whenever $(A,B,C)$ depend on $N$ and equal $(N, N+ u N^{1-q}, N^{(1-q)/2})$ for some $q\in (0,1)$ and $u \in (-1,1)$.

\begin{lem}\label{lem:quantitative} Fix $q\in (0,1)$, and $u \in(-1,1)$. Let
$\psi$ be given by $\eqref{eq:psi}$ with $(a,b,c)=
(a_N, b_N, c_N)= (N, N+u N^{1-q}, N^{(1-q)/2})$. Then there is a number $C(k,q)<\infty$ such that if $N\geq 2$ and $\epsilon_N = C(k,q) N^{-q}$, then 
\[
 \vert \psi^{-1}(z) - z \vert \leq \epsilon_N \vert z \vert, \,
 \Vert D\phi^{-1}(z) - \Id \Vert \leq \epsilon_N, \,
 \vert 1 - \vert \det D\psi^{-1}(z) \vert \vert \leq \epsilon_N
\]
Taking $C(k,q)=\frac{10+2k(1+2k/(1-2^{-q})^{k/2})}{1-2^{-q}} $ suffices.
\end{lem}

\begin{rmk}
In fact, it can be shown that
\begin{eqnarray*}
 \vert \psi^{-1}(z) - z \vert & \leq & \frac{4}{1-2^{-q}} N^{-q} \vert z \vert \\
 \Vert D\phi^{-1}(z) - \Id \Vert &\leq&  \left(4\frac{(1+2^{-q})}{1-2^{-q}}+1 \right) N^{-q}\\
 \left\vert 1 - \vert \det D\psi^{-1}(z) \vert \right\vert & \leq & \frac{1+2k\left( 1+\frac{4N^{-q}}{1-2^{-q}}\right)^{k/2} }{1-2^{-q}}N^{-q}
\end{eqnarray*}
\end{rmk}

\begin{proof}[Proof of Lemma \ref{lem:quantitative}]
The proof follows by studying how the coefficients of $z$ in \eqref{eq:psi_inverse} and of the matrices in \eqref{eq:Dpsi_inverse} compare to $1$. It boils down to showing the following three identities which can be proven using elementary computations. In the following $z_0^2 = a_N c_N/(b_N+ c_N)$ as hinted at in \eqref{eq:psi_inverse}.

\begin{itemize}
    \item $\left\vert \frac{b_N+ c_N}{a_N}-1\right\vert \leq 2 N^{-q}$, $N\geq 1$ (Thus, $\left\vert \sqrt{\frac{b_N+ c_N}{a_N}}-1\right\vert \leq 2 N^{-q}$ \\ and $\left\vert \left(\frac{b_N+ c_N}{a_N}\right)^l-1\right\vert \leq 2 l N^{-q} (1+2 N^{-q})^{l-1}$, for all $l >1$)
    \item $\left\vert \frac{b_N}{a_N -\vert z\vert^2}-1\right\vert \leq \frac{4}{1-2^{-q}}N^{-q}$, whenever $z^2 \leq z_0^2$ \\
    and $N\geq 2$. (Thus $\left\vert \sqrt{\frac{b_N}{a_N -\vert z\vert^2}}-1\right\vert \leq \frac{4}{1-2^{-q}}N^{-q}$ and \\ $\left\vert \left(\frac{b_N}{a_N -\vert z\vert^2}\right)^l-1\right\vert \leq \frac{4}{1-2^{-q}} l N^{-q} \left(1+ 4 N^{-q}/(1-2^{-q}\right)^{l-1}$ for all $l>1$ ).
      \item $\frac{z^2}{a_N - z^2} \leq \frac{N^{-(1+q)/2}}{1-2^{-q}}$, whenever $z^2 \leq z_0^2$, and $N\geq 2$.
      \qedhere
\end{itemize}
\end{proof}
We now state without proof a result of von Bahr and Esseen \cite{vonBahr-esseen1965}.

\begin{lem}[A Quantitative Law of Large Numbers]\label{lem:vonBahr}
Let $Y_1, \dots, Y_N$ be i.i.d. random variables with $\EE Y_1=0$ and $\EE \vert Y_1\vert^s < \infty$, where $1\leq s \leq 2$. Then
\[
	\EE \left[N^{-s} \vert Y_1+\dots+Y_N\vert^s \right] \leq 2 N^{-(s-1)} \EE \vert Y_1 \vert^s.
\]
\end{lem}

\begin{cor}\label{cor:vonBahr} Let $X_1, \dots, X_{N}$ ($N>k$) be i.i.d.,  $ \EE \vert X_1 \vert^{2+\delta} < \infty$, and let $q \in (0,1)$. Define $\tilde{S}_{N,k}
$ as follows:
\[
\tilde{S}_{N,k}= \left\{ \vert X_1^2+ \dots + X_{N-k}^2- N \EE X_1^2  \vert >N^{1-q}\right\}.
\]
Then, whenever $N \geq \max\left\{ 2k , (2k \EE X_1^2)^{1/(1-q)}\right\}$, we have:
\begin{equation}\label{eq:SNk-estimate}
 \PP(\tilde{S}_{N,k}) \leq 16 N^{-(\bar{\delta}/2-(1+\bar{\delta}/2)q)} \EE \vert X_1\vert^{2+\bar{\delta}}.
\end{equation}
Here $\bar{\delta}= \min\{2, \delta \}$.
\end{cor}

\begin{proof}
Without loss of generality, we can let $\delta = \min\{ \delta, 2\}$. Let the $Y_j$ be as in Lemma \ref{lem:vonBahr}, And let $M_j$ be given by
\[
    M_j= \frac{Y_1+ \dots+ Y_j}{j}.
\]
It follows from Tchebyshev's inequality and Lemma \ref{lem:vonBahr} that for any $\beta_N>0$, we have
\begin{equation}\label{eq:Tchebyshev1}
	\PP\left( \vert M_N \vert > \beta_N \right)\leq 2 N^{1-s}\beta_N^{-s}\EE \vert Y_1 \vert^s.
\end{equation}
We will let $Y_j = X_j^2 - \EE X_j^2$ and $s= (2+\delta)/2$. Thus,
\[
    \EE \vert Y_1\vert^s \leq 2^{1+\delta/2} \EE \vert X_1 \vert^{2+\delta}. 
\]
We now note that

\begin{align*}
	\PP \left(\tilde{S}_{N,k}^c \right)
	&= \PP\left( \left\vert M_{N-k}- \frac{k}{N-k}\EE X_1^2 \right\vert
	< \frac{N^{1-q}}{N-k} \right) \\
    &= \PP\left( \frac{k}{N-k}\EE X_1^2 -\frac{N^{1-q}}{N-k}
    < M_{N-k}
    < \frac{k}{N-k}\EE X_1^2 + \frac{N^{1-q}}{N-k} \right) \\
    &\geq  \PP\left( \frac{\frac{1}{2} N^{1-q}}{N-k} -\frac{N^{1-q}}{N-k}
    < M_{N-k}
    < \frac{N^{1-q}}{N-k} \right) \\
    &\geq \PP\left( \left\vert M_{N-k} \right\vert
	< \frac{1}{2} \frac{N^{1-q}}{N-k} \right) \geq \PP\left( \left\vert M_{N-k}\right\vert
	< \frac{1}{2} N^{-q} \right).
\end{align*}
In the last two steps above we used the assumptions $N \geq (2k \EE X_1^2)^{1/(1-q)}$ and $N \geq 2k$. Using \eqref{eq:Tchebyshev1} with $s = 1+\delta/2$ we obtain:
\[
\PP\left(\tilde{S}_{N,k}\right) \leq 2^{3+\delta/2} N^{-(\delta/2-(1+\delta/2)q)} \EE \vert X_1\vert^{2+\delta},
\]
which gives \eqref{eq:SNk-estimate}.
\end{proof}

The following two Lemmas are key to proving  Theorem \ref{thm:POC_L1}. Lemma \ref{lem:L1-comparison} shows in a quantitative way that for a Lipschitz, $L^1$ function $g$ and a homeomorphism $\phi_N$ of $\RR^k$, which is close to the identity map, the function $g_N(x) = g(\phi_N(x)) \vert \det D\phi_N(x) \vert$ remains close in $L^1$ to $g$. Lemma \ref{lem:quantitative} relaxes the Lipschitz assumption and allows us to take functions in $\ALip(r)$.

\begin{lem}\label{lem:L1-comparison} Let $g \in L^1(\RR^k)$ be Lipschitz and satisfy the following tail bounds for some $\mathcal{M}>0, \beta >0$:
\[
    \int_{B(0,R)^c} \vert g(x)\vert dx \leq \mathcal{M} R^{-\beta}.
\]
Let $\{\phi_N: \RR^k \to \RR^k \}$ be a sequence of continuous, $1-1$ and onto functions, have a (matrix-valued) derivatives $D\phi_N$ defined a.e., and
\begin{equation}\label{eq:phi_N_identity}
\vert \phi_N(x) - x\vert\leq \epsilon_N \vert x \vert \text{ a.e. }x, \Vert D\phi_N - \Id \Vert \leq \epsilon_N
\end{equation}
for some $\epsilon_N$ which approaches $0$. Define $g_N$ via
\[
 	g_N(x)= g(\phi_N(x)) \vert \det D\phi_N(x)\vert.
\]
Then we can quantitatively show that $\lim_{N\to \infty} \Vert g - g_N \Vert_{L^1}=0$:
\begin{equation}\label{eq:small_perturbation_of_x2}
\Vert g - g_N\Vert_{L^1(\RR^k)} \leq \begin{array}{l}
  (k+1+\beta) \left( \frac{2 \mathcal{M} }{(k+1)(1-\epsilon_N)^{k+\beta}}\right)^{\frac{k+1}{k+1+\beta}} \left( \frac{  \Vert g\Vert_{\text{Lip}}  \frac{\vert S^{k-1}(1)\vert}{k+1}}{\beta}\right)^{\frac{\beta}{k+1+\beta}}  \epsilon_N^{\frac{\beta}{k+1+\beta}}\\
+ \frac{ (1+\epsilon_N)^k-1}{ (1-\epsilon_N)^k} \int_{\RR^k}\vert  g(y) \vert \, dy \end{array}
\end{equation}
\end{lem}

In \eqref{eq:phi_N_identity} $\vert \cdot \vert$ denotes the Euclidean norm on $\RR^k$ and $\Vert \cdot \Vert$ the associated Matrix norm, given by $\Vert M \Vert = \sup\{ \vert M x \vert: \vert x\vert=1\}$.

\begin{proof}
Let $T_R= \int_{B(0,R)^C} \vert g(x)\vert \, dx$. Then we have the following:
\begin{eqnarray*}
	\Vert g - g_N\Vert_{L^1(\RR^k)} & \leq & \int_{\RR^k} \vert g(x)-g( \phi_N(x))\vert \,dx+ \int_{\RR^k}\vert  g( \phi_N(x))\vert   \vert 1- \vert  \det D\phi_N(x) \vert \vert\\
	& \leq &  \int_{\RR^k} \vert g(x)-g( \phi_N(x))\vert \,dx+ \sup_{\RR^k} \frac{\vert 1 -\vert \det D\phi_N(x)\vert \vert}{\vert \det D\phi_N(x)\vert} \int_{\RR^k} \vert  g(y) \vert  \, dy \\
	& \leq &   \int_{\RR^k} \vert g(x)-g( \phi_N(x))\vert \,dx+ \frac{ (1+\epsilon_N)^k-1}{ (1-\epsilon_N)^k} \int_{\RR^k} g(y) \, dy \\
	& = & \int_{B(0,R)} \vert g(x)-g( \phi_N(x))\vert \,dx + \int_{B(0,R)^c} \vert g(x)-g( \phi_N(x))\vert \,dx  \\
		&   &  + \frac{ (1+\epsilon_N)^k-1}{ (1-\epsilon_N)^k} \int_{\RR^k} \vert  g(y) \vert  \, dy \\
	&  \leq &  \Vert g\Vert_{\text{Lip}} \epsilon_N \int_{B(0,R)} \vert x \vert\, dx + T_R + \frac{1}{(1-\epsilon_N)^k} T_{(1-\epsilon_N) R} \\
	&	& +  \frac{ (1+\epsilon_N)^k-1}{ (1-\epsilon_N)^k} \int_{\RR^k} \vert  g(y) \vert  \, dy \\
	&\leq &    \Vert g\Vert_{\text{Lip}}  \frac{\vert S^{k-1}(1)\vert}{k+1}  \epsilon_N R^{k+1} +  \frac{2}{(1-\epsilon_N)^k} T_{(1-\epsilon_N) R} \\
	&	& +   \frac{ (1+\epsilon_N)^k-1}{ (1-\epsilon_N)^k} \int_{\RR^k}\vert  g(y) \vert \, dy \\
\end{eqnarray*}
To show \eqref{eq:small_perturbation_of_x2}, we substitute 
$T_{(1-\epsilon_N R)}= \mathcal{M} (1-\epsilon_N)^\beta R^\beta$ and
choose 
\[
R^{k+1+\beta} = 
\frac{2 \mathcal{M}\beta }{ \Vert g\Vert_{\text{Lip}} \vert S^{k-1}(1)\vert (1-\epsilon_N)^{k+\beta} \epsilon_N}.
\qedhere
\]
\end{proof}

\begin{lem}\label{lem:L1-comparison-non-Lipschitz} Let $k,\phi_N, \epsilon_N$ be as in Lemma \ref{lem:L1-comparison}.
Let $f \in \ALip(r, L_0)$ for some $(r, L_0)$ satisfy
\[
    \int_{B(0,R)^c} \vert f(x)\vert dx \leq \mathcal{M} R^{-\beta},
\]
and let $f_N(x)= f(\phi_N(x)) \vert \det D\phi_N(x)\vert.$
Then
\[
	\Vert f - f_N\Vert_{L^1(\RR^k)} \leq C(k, \beta, \mathcal{M}, L_0) \epsilon_N^{1/(1+r + (k+1)/\beta)}.
\]
The constant can be made explicit.
\end{lem}

\begin{proof}
Given $\epsilon$, we can choose $g$ as in the hypothesis, which satisfies the same tail bound $T_R$ up to a factor of $2$. By a $3\epsilon$ argument, we have 
\[
\Vert f - f_N\Vert_{L^1(\RR^k)} \leq 2\epsilon + \left( \begin{array}{c} \text{r.h.s. of \eqref{eq:small_perturbation_of_x2}}\\ \text{ with $2\mathcal{M}$} \\ \text{and  $ \Vert g\Vert_{\text{Lip}} = L_0 \epsilon^{-r}$}\end{array}  \right)
\]
and we choose $\epsilon \sim \epsilon_N^{1/(1+r + (k+1)/\beta)}$.
\end{proof}

\begin{lem}
\label{lem:exponents}
Fix $k \in \mathbb{N}$ and $\delta \in (0, 2]$. Let $\eta_1(q), \eta_2(q)$, and $\eta_3(q)$ be given by
\[
   \eta_1(q)= \frac{2+\delta}{k+ 3+\delta}q, \, \eta_2(q) = \frac{\delta}{2} - q \left(1+ \frac{\delta}{2} \right), \text{ and } \eta_3(q) = \frac{2+\delta}{4} (1-q).
\]
Then
\begin{equation}\label{eq:best_exponent}
        \max_{q\geq 0} \left[ \min\left(\eta_1(q),\eta_2(q),\eta_3(q) \right) \right] 
         =  \frac{\delta}{k+5+\delta+r(2+\delta)}
\end{equation}
which is $\eta_1(q_\ast)$ with
\begin{equation}\label{eq:best_q}
    q_\ast = \left(\frac{\delta}{\delta+2}\right) \left(\frac{k+3+\delta+r(2+\delta)}{ k+5+\delta + r(2+\delta)}\right)
\end{equation}
\end{lem}

\begin{proof}
Since $\delta \leq 2$, it is easy to show $\eta_2(q) \leq \eta_3(q)$ on $[0, 1]$. Thus, $\eta_3(q)$ can be neglected in the left hand side of \eqref{eq:best_exponent}. We are left with maximizing the minimum of two lines: one increasing and the other decreasing. Thus, this maximum is at their intersection. This provides $q_\ast$ in \eqref{eq:best_q} and the maximum in \eqref{eq:best_exponent}.
\end{proof}

We now show how non-smooth can the functions in $\ALip(r)$ be, by computing the exponent $r$ in the cost in the Lipschitz in the definition $\ref{defi:ALip}$ for some functions in $L^1_{\text{loc}}(\RR)$.

\begin{lem}\label{lem:almost-Lipschitz} The following table gives the cost in terms of the Lipschitz constant, of $\epsilon$-approximating  some functions using the $L^1$ metric.

\begin{tabular}{|l|c|}
\hline
function & $\Vert g_\epsilon\Vert_{\text{Lip}}$ (upper bound)\\
\hline 
$   \vert x\vert^a$, \, $a \in (-1,1]$ & $C(a) \epsilon^{- \left(\frac{1-a}{1+a}\right)}$ \\
$\mathbf{1}_{[0, \infty)}$     & $\frac{1}{4}\epsilon^{-1}$ \\
$f \in C^\alpha \cap L^1( (1+\vert x\vert^\beta)dx)$ & $C(f) \epsilon^{-\left(\frac{1+\beta}{\alpha \beta}\right)}$\\
\hline
\end{tabular}
\end{lem}

\begin{proof}
We show the proofs of some of them. 
\begin{enumerate}
	\item When $f(x)=\vert x \vert^{a}$, with $a \in (-1,0)$. Set
	\[f_h(x)=\left\{ \begin{array}{cc} \vert x \vert^{a}, & \vert x\vert \geq h \\ h^{a}+a h^{a-1} (x-h), & 0 \leq \vert x\vert  \leq r \end{array}\right..\]
Then $\Vert f - f_h \Vert_{L^1} = 2a \left[ \frac{1}{1-a}+ \frac{1}{2}\right] h^{1-a}=:\epsilon$, while $\Vert f_h \Vert_{\text{Lip}}=2 h^{-(1+a)} = C(a) \epsilon^{-\frac{1+a}{1-a}}$.
If $a >0$, the formula for $f_h$ above need not change.
	\item When $f(x)=\mathbf{1}_{[0,\infty)}$, we set $f_h(x)=\left\{ \begin{array}{cc} 1, & x>h/2,\\ \frac{h}{2}+ \frac{2-h}{h}x, & -h/2< x < h/2\\ 0, & \text{ otherwise}  \end{array}\right.$
	
Then $\Vert f - f_h \Vert_{L^1} = (h/4)=:\epsilon$, while $\Vert f_h \Vert_{\text{Lip}}=(1/h)$.

    \item Let $f$ be H\"{o}lder continuous with order $\alpha$, and have tail bounds
    \[
        \int_{[-R, R]^c}\vert f \vert dx \leq \mathcal{M} R^{-\beta}.
    \]
   It easily follows that $\Vert f\Vert_{L^\infty} <\infty$, and 
   an elementary computation shows that, whenever $\vert h\vert \leq 1$, we have
\[
	\Vert \tau_h f - f \Vert_{L^1} \leq 2\max\left\{ \Vert f \Vert_\alpha,  \Vert f \Vert_\alpha^{\frac{\beta}{1+\beta}}\right\}\left[ \frac{1}{\beta}+ (\beta \mathcal{M})^{\frac{1}{1+\beta}}\right] \vert h\vert^{\frac{\beta}{1+\beta} \alpha}.
\]
Here $\tau_h f(x)$ is the translate of $f$ ($=f(x-h)$).
Using this inequality, choose $\psi \geq 0 \in C^1(\RR)$ that satisfies $\int_\RR \psi=1$ and supp$(\psi) \subset [-1,1]$. Let $\psi_\delta(x)=\delta^{-1} \psi(x/\delta)$. Then, for any $\delta>0$, $f\ast \psi_\delta$ is Lipschitz and we have:
\begin{eqnarray*}
	\Vert f \ast \psi_\delta - f \Vert_{L^1} & \leq & \sup_{\vert h\vert \leq \delta} \Vert \tau_h f - f \Vert_{1}=  C\left( \Vert f \Vert_{C^\alpha}, \mathcal{M}, \beta \right) \delta^{\frac{\beta \alpha}{1+\beta}},\\
	\Vert ( f \ast \psi_\delta)' \Vert_{L^\infty} & \leq & \int_\RR \vert f(x-y)\delta^{-2} \psi'( y/\delta) \vert dy \leq \delta^{-1} \Vert \psi' \Vert_{L^1} \Vert f \Vert_{L^\infty}.
\end{eqnarray*}
Hence, choosing $\delta$ so that $\Vert f \ast \psi_\delta - f \Vert_{L^1} =\epsilon$, gives $\Vert ( f \ast \psi_\delta)' \Vert_{L^\infty} \leq C \epsilon^{-(1+\beta)/(\alpha \beta)}$. \qedhere
\end{enumerate}
\end{proof}

\begin{lem}\label{lem:Besov} For any $r>0$ and $\beta>0$ the following inclusion holds for $\ALip(r)$:
\[ 
    \ALip(r) \cap \left\{f: \int_\RR \vert x \vert^\beta \vert f(x)\vert dx<\infty \right\} \subseteq B_{1,\infty}^q(\RR),
\]
where $q=\frac{\beta}{1+\beta(1+r)}$ and $B_{1,\infty}^q(\RR)$ is a Besov space\footnote{We recall that $B_{a,\infty}^q$ can be characterized as the space of functions $f \in L^a$ such that 
\[
    \sup_{h>0} h^{-q} \Vert f - \tau_h f \Vert_{L^a} < \infty.
\] }.
\end{lem}
\noindent The proof is omitted, since it is standard and involves computations similar to those in the proof of \ref{lem:almost-Lipschitz}.

We close this section by giving a remark concerning the strength of our POC result in $L^1$, whenever it applies, to the weak-$L^1$ propagation of chaos result for $\{\hat{f}^N \}_N$ which can be obtained via the Dunford-Pettis theorem whenever $H[f\vert \gamma] <\infty$, as hinted at in \cite{carlen-carvalho-leroux-loss-villani2010}.

\begin{rmk} If $H[f\vert \gamma]<\infty$ holds, it will follow from Theorem  \ref{thm:entropic_chaos_rescaled_tensor} below that
\[ 
    \sup_N N^{-1} H[\hat{f}^N \vert \sigma^N] <\infty.
\] As remarked in \cite{carlen-carvalho-leroux-loss-villani2010}, this implies that $\sup_N H(\Pi_k \hat{f}^N \vert \gamma^{\otimes k})< \infty$ (For the proof, see \cite{carlen-lieb-loss2004} when $k=1$ and see \cite{barthe-cordero-erausquin-maurey2006} for general $k$). Thus, the sequence $\{\Pi_k \hat{f}^N\}_{N\geq k+1}$ is uniformly integrable and the Dunford-Pettis theorem implies that $\{\Pi_k \hat{f}^N\}_{N\geq k+1}$ is compact in the weak topology. Thus, $\Pi_k \hat{f}^N$ weakly converges to $f^{\otimes k}$ in $L^1$ (i.e. against all $\phi \in L^\infty(\RR^k)$ and not just for $\phi$ continuous and bounded). Theorem \ref{thm:POC_L1}, which holds for $f \in \ALip(r)$, is quantitative and gives convergence in the strong $L^1$ norm.
\end{rmk}

\section{Entropic chaos}
\label{sec:entropic_chaos}

We now study entropic chaos for rescaled tensor products. Let us first recall Theorem 12 in \cite{carlen-carvalho-leroux-loss-villani2010} (see also \cite[Theorem $1.15$]{carrapatoso-einav2013}). For convenience, we provide a proof in the Appendix. Our proof will avoid the conditioned states and the associated local version of the central limit theorem used in \cite{carlen-carvalho-leroux-loss-villani2010}, and rely instead on rescaling and the classical central limit theorem.

\begin{thm}[asymptotic upper semi-extensivity of the entropy]
\label{thm:semi-extens}
For each $N\in\NN$, let $F^N \in \Psym(S^{N-1})$ be such that $\lim_N \Pi_1 F^N = f$ weakly, for some $f\in\P(\RR)$ satisfying $H(f \mid \gamma) < \infty$. Then,
\begin{equation}
    \label{eq:liminfHFN}
    H(f \mid \gamma)
    \leq \liminf_{N\to\infty} \frac{1}{N} H(F^N \mid \sigma^N).
\end{equation}
\end{thm}

In view of \eqref{eq:liminfHFN}, we see that, in order to prove that an $f$-chaotic sequence $F^N$ is also $f$-entropically chaotic, it suffices to show that $\limsup_N N^{-1} H(F^N \mid \sigma^N) \leq H(f \mid \gamma)$. We will use this strategy to prove that $\hat{f}^N$ is entropically chaotic to $f$, which is the main goal of this section. We will need the following formula for the density of the rescaled measure:


\begin{prop}[formula for rescaled densities]
\label{prop:formula_rescaled_density}
Let $F^N \in \P(\RR^N)$ have a density. Then $\hat{F}^N \ll \sigma^N$, and we have:
\begin{eqnarray}
    \frac{d\hat{F}^N}{d\sigma^N}(x) &
= & \sqrt{N} \left\vert S^{N-1} \right\vert  \int_0^\infty r^{N-1} F^N(rx) dr, \nonumber \\
 & = & \left\vert S^{N-1}(1) \right\vert \int_0^\infty r^{N-1} F^N( r x/\vert x\vert ) dr \qquad \forall x\in S^{N-1}. \label{eq:F_hat_N_formula}
\end{eqnarray}
\end{prop}

\begin{proof}
For any test function $\phi:S^{N-1} \to \RR$, using polar coordinates, we have:
\begin{align*}
    \int_{S^{N-1}} \phi(x) \hat{F}^N(dx)
    &= \int_{\RR^N} \phi(\hat{x}) F^N(x) \, dx \\
    &= \int_{S^{N-1}(1)} \int_0^\infty r^{N-1} \phi(\sqrt{N} \omega) F^N(r\omega) \, dr \, d\omega \\
    &= (\sqrt{N})^N \int_{S^{N-1}(1)}  \phi(\sqrt{N} \omega) \int_0^\infty u^{N-1} F^N(u \sqrt{N} \omega) \, du \, d\omega \\
    &= \sqrt{N} \int_{S^{N-1}}  \phi(y) \int_0^\infty u^{N-1} F^N(u y) \, du \, dy,
\end{align*}
where we have used the changes of variables $r = \sqrt{N} u$ and $y = \sqrt{N} \omega$. The conclusion now follows simply noting that $dy = \vert S^{N-1}\vert  \sigma^N(dy)$. The formula in \eqref{eq:F_hat_N_formula} follows from a simple rescaling.
\end{proof}

In Lemma \ref{lem:HhatF-HF} below, we write the difference between $H(\hat{F}^N \mid \sigma^N)$ and $H(F^N \mid \gamma^{\otimes N})$ as the negative of some relative entropy with respect to the following distribution:

\begin{defi}[angular version]
Let $F^N \in \P(\RR^N)$ without an atom at 0. We define an \emph{angular version} of $F^N$, denoted $\check{F}^N \in \P(\RR^N)$, as the law of $\vert Z\vert  \hat{X} / \sqrt{N}$, where $X \sim F^N$ and $Z\sim \gamma^{\otimes N}$ are independent.
\end{defi}

The following lemma provides a formula for the density of the angular version in terms of the density of the rescaled measure:

\begin{lem}[formula for the density of the angular version]\label{lem:density-formula-angular-version}
Let $F^N \in \P(\RR^N)$ be such that $\hat{F}^N$ has a density with respect to $\sigma^N$. Then $\check{F}^N$ has a density with respect to the Lebesgue measure, given by
\begin{equation}
\label{eq:checkFN}
\check{F}^N(x) = \gamma^{\otimes N}(x) \frac{d\hat{F}^N}{d\sigma^N}(\hat{x}),
\qquad \forall x\in\RR^N.
\end{equation}
\end{lem}

\begin{proof}
For any test function $\phi: \RR^N \to \RR$, we have
\begin{align*}
\int_{\RR^N} \phi(x) \check{F}^N(dx)
&= \int_{\RR^N} \int_{S^{N-1}} \phi \left( \frac{\vert z\vert  y}{\sqrt{N}} \right) \hat{F}^N(dy) \gamma^{\otimes N}(z) \, dz \\
&= \int_{S^{N-1}(1)} \int_0^\infty \int_{S^{N-1}} \phi \left( \frac{r y}{\sqrt{N}} \right) \hat{F}^N(dy) \gamma^{\otimes N}(r \omega) r^{N-1} \, dr \, d\omega,
\end{align*}
where we have changed from Cartesian $z\in\RR^N$ to polar coordinates $(r,\omega) \in (0,\infty) \times S^{N-1}(1)$. The key step is to note that $\gamma^{\otimes N}(r \omega) = \gamma^{\otimes N}(r y/\sqrt{N})$ for any $y\in S^{N-1}$ and $\omega \in S^{N-1}(1)$. With this, and noting that then the integrand does not depend on $\omega$, we obtain:
\[
\int_{\RR^N} \phi(x) \check{F}^N(dx)
= \left\vert S^{N-1}(1)\right\vert  \int_0^\infty \int_{S^{N-1}} \phi \left( \frac{r y}{\sqrt{N}} \right) \hat{F}^N(dy) \gamma^{\otimes N}\left( \frac{r y}{\sqrt{N}} \right) r^{N-1} \, dr.
\]
Since $\hat{F}^N$ has a density with respect to $\sigma^N$, we have
\[
\hat{F}^N(dy)
= \frac{d \hat{F}^N}{ d\sigma^N}(y) \sigma^N(dy)
= \frac{d \hat{F}^N}{ d\sigma^N}(\sqrt{N} \omega) \frac{d\omega}{\vert S^{N-1}(1)\vert },
\]
for $\omega = y/\sqrt{N} \in S^{N-1}(1)$. Consequently,
\begin{align*}
\int_{\RR^N} \phi(x) \check{F}^N(dx)
&= \int_0^\infty \int_{S^{N-1}(1)} \phi(r\omega) \frac{d \hat{F}^N}{ d\sigma^N}(\sqrt{N} \omega) \, d\omega \, \gamma^{\otimes N}(r\omega) r^{N-1} \, dr \\
&= \int_{\RR^N} \phi(x) \frac{d \hat{F}^N}{d \sigma^N} (\hat{x}) \gamma^{\otimes N}(x) \, dx,
\end{align*}
where we changed back to Cartesian coordinates $x = r \omega \in \RR^N$. The result follows.
\end{proof}


\begin{lem}
\label{lem:HhatF-HF}
Let $F^N$ be a probability density on $\RR^N$ such that $H(F^N \mid \gamma^{\otimes N}) < \infty$. Then,
\begin{equation}
\label{eq:HcheckF}
H(\hat{F}^N \mid \sigma^N) - H(F^N \mid \gamma^{\otimes N})
= - H(F^N \mid \check{F}^N )
\leq 0.
\end{equation}
\end{lem}

\begin{proof}
We have:
\begin{align*}
H( \hat{F}^N \mid \sigma^N) - H( F^N \mid \gamma^{\otimes N})
& = \int_{\RR^N} \log \left( \frac{d \hat{F}^N}{d\sigma^N}(\hat{x}) \right) F^N(x) dx
- \int_{\RR^N} \log \left( \frac{F^N(x)}{\gamma^{\otimes N}(x)} \right) F^N(x) dx \\
&= \int_{\RR^N} \log \left( \frac{\check{F}^N(x)}{F^N(x)} \right) F^N(x) dx,
\end{align*}
where we have used \eqref{eq:checkFN}. The last expression equals $-H(F^N \mid \check{F}^N)$.
\end{proof}

We can now state and prove one of our main results: entropic chaoticity for rescaled tensor products, under the minimal assumptions of unit energy and finite entropy relative to $\gamma$.

\begin{thm}[entropic chaos for rescaled tensor products]
\label{thm:entropic_chaos_rescaled_tensor}
Let $f$ be a probability density on $\RR$ such that $\int_\RR x^2 f(x) dx = 1$ and $H(f \mid \gamma) < \infty$. Then $\hat{f}^N$ is entropically chaotic to $f$.
\end{thm}

\begin{proof}
Since $H(f^{\otimes N} \mid \gamma^{\otimes N}) = N H(f \mid \gamma)$, taking $F^N = f^{\otimes N}$ in \eqref{eq:HcheckF} gives $\frac{1}{N}H(\hat{f}^N \mid \sigma^N) \leq H(f \mid \gamma)$. Taking $\limsup_N$ and using \eqref{eq:liminfHFN}, the result follows (recall that, since $\int_\RR x^2 f(x) dx = 1$, we know that $\hat{f}^N$ is $f$-chaotic, thanks to Theorem \ref{thm:chaos_rescaled_tensor}).
\end{proof}

The last result of this section provides a quantitative entropic chaos rate for $\hat{f}^N$, see Theorem \ref{thm:quantitative_entropic_chaos_for_rescaled_tensor_products} below. The proof relies on Proposition \ref{prop:I_N_inequality} provided in the next section, and on the following result, which is a slight improvement of \cite[Theorem 4.17]{hauray-mischler2014}:

\begin{thm}
\label{thm:H&M_JFA2014}
Let $F^N \in \Psym(S^{N-1})$ be an $f$-chaotic sequence, for some $f\in \mathcal{P}(\RR) \cap L^p(\RR)$ for some $p>1$. Assume that there exists $M > 0$ such $\int_{\RR^N} \vert x_1\vert ^k F^N(dx) \leq M$ for some $k > 4$, and that $\frac{1}{N} I(F^N \mid \sigma^N) \leq M$, for all $N$. Then $F^N$ is entropically chaotic to $f$, with the following explicit rate:
\begin{equation}
\label{eq:lem_H&M_JFA2014}
    \left\vert  \frac{1}{N} H( F^N \mid \sigma^N) - H(f \mid \gamma) \right\vert 
    \leq C_1 \left( N^{-1/2} W_2(F^N, f^{\otimes N} ) + C_2 N^{-\eta} + C_3 N^{-(k/4-1)}\right),
\end{equation}
for any $\eta < \frac{1}{8} \frac{k-2}{k+1}$, where $C_1 = C_1(M), C_2 = C_2(\eta)$, and $C_3=C_3(p, \Vert f\Vert_p,  k, M)$.
\end{thm}

The authors in \cite{hauray-mischler2014} require at least $6$ finite moments on $f$ and, although they present their theorem in terms of a normalized $W_1$ metric, they provide  the above result for $W_2$ in their proof. We manage to relax this condition to $k >4$ finite moments. To achieve this, we analyze the method in their proof. We will need to define the conditioned tensor product states.
   
   \begin{defi}\label{def:conditioned-states} Let $f \in L^1(\RR)$. Then the conditioned state $[f^{\otimes N}]_N $ is the element in $L^1(S^{N-1}, \sigma^N)$ given by
   \[
    [f^{\otimes N}]_N(v) = \frac{ f^{\otimes N}(v) }{\int_{S^{N-1}} f^{\otimes N}(y)\, \sigma^N(dy)}.
    \]
   \end{defi}
We note that when $f \in L^1(\RR)$, the product structure of $f^{\otimes N}$ makes $[f^{\otimes N}]_N$ well defined, in spite of the denominator in its definition depending on the values of $f^{\otimes N}$ on a set of measure zero (see the note in the beginning of the proof of Theorem $4.9$ in \cite{hauray-mischler2014}).

The proof of Theorem \cite[Theorem 4.17]{hauray-mischler2014} relies on two results concerning the conditioned tensor product states.
The first one concerns their entropic chaoticity: when $f \in \mathcal{P}(\RR) \cap L^p(\RR) \cap \P_6(\RR)$, we have
\begin{equation}\label{eq:conditioned-entropic-POC-old}
\left\vert  \frac{1}{N} H( [f^{\otimes N}]_N \mid \sigma^N) - H(f \mid \gamma) \right\vert  \leq C(p, \Vert f\Vert_p, M)N^{-1/2}.
\end{equation}
The second result gives a uniform bound on their Fisher information: if $f \in \mathcal{P}(\RR) \cap L^p(\RR) \cap \P_6(\RR)$ and $I(f\vert \gamma)<\infty$, then
$\sup_N \frac{1}{N} I([f^{\otimes N}]_N \vert \sigma^N)$ (see Theorem $4.14$ in \cite{hauray-mischler2014}). In the following two lemmas we extend these results to the case $f \in \mathcal{P}_k(\RR)$ with $k=4+r$ for some $r \in (0,2]$.

\begin{lem}\label{lem:conditioned-product-improvement}
Let $f\in \mathcal{P}_{4+r}(\RR)\cap L^p(\RR)$ for some $r \in (0,2], p>1$, $\int_\RR v^2 f(v)\,dv=1$. Then
\begin{equation}\label{eq:entropic_chaos_conditioned_estates}
\left\vert  \frac{1}{N} H( [f^{\otimes N}]_N \mid \sigma^N) - H(f \mid \gamma) \right\vert  \leq C\left(p, \Vert f\Vert_p, \int_\RR \vert v\vert^{4+r} f(v)\,dv \right) N^{-r/4}.
\end{equation}
\end{lem}

We note that we do not require $\int_\RR v f(v)\, dv$ to be $0$. The proof of Lemma \ref{lem:conditioned-product-improvement} requires only small adjustments to that of \eqref{eq:conditioned-entropic-POC-old}, which we list in the Appendix.

\begin{lem}\label{lem:conditioned-product-bounded-fisher-information} Let $f\in \mathcal{P}(\RR)\cap L^p(\RR)$ for some $p>1$, $\int_\RR \vert v\vert^{2} f(v)\,dv =1$, and assume that $I(f\vert \gamma)< \infty$. Then
\[
    \sup_{N \geq 2} \frac{1}{N} I([f^{\otimes N}]_N \vert \sigma^N) < \infty.
\]
\end{lem}

Lemma \ref{lem:conditioned-product-bounded-fisher-information} follows from the proof of the first statement in \cite[Theorem $4.14$]{hauray-mischler2014}, and noting that we do not need the higher moment assumptions on $f$ for this part. For completeness, we sketch the proof in the Appendix. We are now ready to prove Theorem \ref{thm:H&M_JFA2014}, which is an adaptation to that of \cite[Theorem 4.17]{hauray-mischler2014}.

\begin{proof}[Proof of Theorem \ref{thm:H&M_JFA2014}]
Let $F^N$ and $f$ be as in the hypotheses. A version of the HWI inequality on spaces of positive Ricci curvature, proven in \cite{otto-villani2000} (see also \cite[Theorem 30.21] {villani2009}), implies that 
\[
    \frac{1}{N}\left\vert H(F^N \vert \sigma^N)- H([f^{\otimes N}]_N \vert \sigma^N) \right\vert \leq \frac{\pi}{2} \sqrt{\frac{\max\{I([f^{\otimes N}]_N,\sigma^N), I(F^N \vert \sigma^N)\}}{N}}\frac{W_2(F^N , [f^{\otimes N}]_N)}{\sqrt{N}}.
\]
Thus, we have that
\begin{eqnarray*}
\left\vert \frac{1}{N} H(F^N \vert \sigma^N) - H(f\vert \gamma) \right\vert & \leq &  \frac{1}{N}  \left\vert H(F^N \vert \sigma^N) - H([f^{\otimes N} \vert \sigma^N) \right\vert + \left\vert \frac{1}{N} H([f^{\otimes N}]_N \vert \sigma^N) - H(f\vert \gamma) \right\vert\\
& \leq & C_1 \frac{W_2(F^N \vert [f^{\otimes N}]_N)}{\sqrt{N}} +  \left\vert \frac{1}{N} H([f^{\otimes N}]_N \vert \sigma^N) - H(f\vert \gamma) \right\vert\\
 &\leq & C_1 \left( \frac{W_2(F^N, f^{\otimes N}) + W_2(f^{\otimes N}, [f^{\otimes N}]_N)}{\sqrt{N}} \right) + \left\vert \frac{1}{N} H([f^{\otimes N}]_N \vert \sigma^N) - H(f\vert \gamma) \right\vert
\end{eqnarray*}
where we used the triangle inequality in the first and last steps, and we used Proposition \ref{prop:I_N_inequality} (below) and Lemma \ref{lem:conditioned-product-bounded-fisher-information} to bound $\frac{1}{N}I(\hat{f}^N \vert \sigma^N)$ and $\frac{1}{N}I([f^{\otimes N}]_N \vert \sigma^N)$. This, together with Lemma \ref{lem:conditioned-product-improvement} and the bound
\[
    \frac{1}{\sqrt{N}} W_2([f^{\otimes N}]_N, f^{\otimes N}) \leq C(\eta) N^{-\eta}
\]
for any $\eta < \frac{1}{8}\frac{k-2}{k+1}$, provided in the proof of \cite[Theorem $4.17$]{hauray-mischler2014},  proves \eqref{eq:lem_H&M_JFA2014}.
\end{proof}

We are now ready to state and prove the main theorem of this section.

\begin{thm}[quantitative entropic chaos for rescaled tensor products]
\label{thm:quantitative_entropic_chaos_for_rescaled_tensor_products}
Let $f\in\P_k(\RR) \cap L^p(\RR)$ for some $k > 4$, $p>1$. Assume that $\int_\RR x^2 f(x) dx = 1$ and $I(f \mid \gamma) < \infty$. Then, $\hat{f}^N$ is entropically chaotic to $f$, with the following explicit rate: for any $\eta < \frac{1}{8} \frac{k-2}{k+1}$, there exists a constant $C>0$ such that
\[
\left\vert  \frac{1}{N} H( \hat{f}^N \mid \sigma^N) - H(f \mid \gamma) \right\vert 
\leq C \left( N^{-\eta}+ N^{-(k/4-1)} \right).
\]
\end{thm}

\begin{proof}
By Theorem \ref{thm:chaos_rescaled_tensor}, $\hat{f}^N$ is $f$-chaotic; by Lemma \ref{lem:moments}, the $k$-th moment of $\hat{f}^N$ is bounded; and by Proposition \ref{prop:I_N_inequality} we know that $\frac{1}{N} I(\hat{f}^N \mid \sigma^N)$ is bounded. Thus, Theorem \ref{thm:H&M_JFA2014} gives \eqref{eq:lem_H&M_JFA2014}.
Thanks to Theorem \ref{thm:chaos_rescaled_tensor},
$N^{-1/2} W_2(\hat{f}^N,f^{\otimes N} )\leq C N^{-1/4}$. Since $\frac{1}{4} > \frac{1}{8} \frac{k-2}{k+1}$ for all $k >0$, the $N^{-1/4}$ term can be neglected. The result follows.
\end{proof}

\section{Fisher information chaos}
\label{sec:fisher_information_chaos}

In this section we prove \textit{Fisher information chaos}, which we now define, for the rescaled states $\{\hat{f}^N\}$.

\begin{defi}[Fisher information chaos]
\label{def:fisher_information}
For each $N \in \NN$, let $F^N \in \Psym(S^{N-1})$, and let $f\in\P(\RR)$ such that $I(f \mid \gamma) < \infty$. The sequence $(F^N)_{N\in\NN}$ is said to be \emph{Fisher information chaotic} to $f$ if it is Kac chaotic to $f$, and
\begin{equation}\label{eq:Fisher-chaos-defi}
    \lim_{N\to \infty} \frac{1}{N} I(F^N \mid \sigma^N ) = I(f \mid \gamma).
\end{equation}
\end{defi}

The main result of this section is the following:

\begin{thm}[Fisher information chaos for rescaled tensor products]\label{thm:fisher_chaos}
Let $f$ be a probability density on $\RR$ such that 
$\int_\RR x^2 f(x) dx=1$ and $I(f \mid \gamma ) < \infty$. Then $\hat{f}^N$ is Fisher information chaotic to $f$.
\end{thm}

To prove this theorem, we need to establish the equality in \eqref{eq:Fisher-chaos-defi}. The new part is in obtaining the limit inequality
\[ 
\limsup_{N\rightarrow \infty}N^{-1} I(\hat{f}^N \vert \sigma^N ) \leq I(f\vert\gamma),
\]
which we will treat in Proposition \ref{prop:I_N_inequality}. The proof uses the convexity of the map $(x,\vec{y}) \mapsto \frac{\vert \vec{y}\vert^2}{x}$. We prove a slightly more general  result in Proposition \ref{prop:fisher-main}. This proposition depends on the technical, but not too difficult, Lemmas \ref{lem:grad_hat_x}-\ref{lem:limit_of_x_f_x}. The reverse inequality, i.e., $\liminf_{N\rightarrow \infty} N^{-1} I(\hat{f}^N \vert \sigma^N ) \geq I(f\vert\gamma)$, has been proven in \cite[Theorem 4.15]{hauray-mischler2014} for sequences more general than $\hat{f}^N$.

\begin{prop}\label{prop:fisher-main} Let  $F \in L^1(\RR^N, d^Nx)$ and $G \in L^1(S^{N-1}, \sigma^N)$ be probability densities; here $N\geq 3$. Then the following hold.
\begin{enumerate}
	\item {\label{item1}} $I(G \vert \sigma^N) = \frac{N-2}{N}I(\check{G}\vert \gamma^{\otimes N})$.
	\item {\label{item2}} $I(\check{\hat{F}} \vert \gamma^{\otimes N}) \leq \frac{1}{(N-2)} \int_{\RR^N} \gamma^{\otimes N}(x) \frac{  \vert x\vert^2 \left\vert \nabla \rho \right\vert^2(x)- (\nabla\rho(x). x)^2 }{ \rho(x) } dx$,
where $\rho(x) = F(x) /\gamma^{\otimes N}(x)$
\end{enumerate}
\end{prop}

\begin{prop}\label{prop:I_N_inequality}
 Let $f$ be a probability density on $\RR$ such that $\int_\RR x^2 f(x) dx = 1$ and $I(f \mid \gamma) < \infty$.
Then for all $N\geq 2$ we have the inequality:
\[
	\frac{1}{N}I(\hat{f}^N \mid \sigma^N) \leq \left( 1 - \frac{1}{N} \right) I(f \mid \gamma)
\]
\end{prop}

\begin{proof}[Proof of Theorem \ref{thm:fisher_chaos}] Proposition \ref{prop:I_N_inequality} implies that $
    \limsup_{N\to \infty} N^{-1}I(\hat{f}^N \vert \sigma^N) \leq I(f\vert \gamma)$.
As mentioned above, the reverse inequality: $\liminf_{N\to\infty} N^{-1} I(\hat{f}^N \vert \sigma^N) \geq I(f\vert \gamma)$ has been proven in \cite[Theorem $4.15$]{hauray-mischler2014}. This establishes the Fisher information chaoticity for the family  $\{\hat{f}^N\}_N$.
\end{proof}

 Before proving Proposition  \ref{prop:fisher-main}, we provide two useful technical lemmas. The first one gives a formula for the gradient of a function $\tilde{J}: \RR^{N} \backslash \{0\} \rightarrow \RR$ that depends only on $x/\vert x\vert$, while the second lemma reduces the integral of such a function integrated against a Gaussian, to an integral on the unit sphere. We omit the proof of the first lemma since it follows directly from the chain rule.

\begin{lem}[gradient of the composition]
\label{lem:grad_hat_x}
Let $J: \RR^N \to \RR$ be $C^1$, let $r>0$  and $x\neq 0$. Let $\tilde{J}$ be the function on $\RR^N \backslash \{0\}$ defined by 
\[
	\tilde{J}(x) = J \left(r\frac{x}{\vert x \vert}\right).
\]
Then, for $x\neq 0$,
\[  \nabla \tilde{J}(x) =  \frac{r}{\vert x\vert}\left( \Id -\frac{x x^T}{\vert x \vert^2}  \right)[\nabla J]\left(r \frac{x}{\vert x \vert}\right)\]
In particular, because the matrix $\Id -  \frac{x x^T}{\vert x\vert^2}$ is a projection, 
\[
	\vert \nabla \tilde{J}(x) \vert^2 =  \frac{r^2}{\vert x\vert^2} \left( \left\vert [\nabla J]\left(r \frac{x}{\vert x \vert}\right)\right\vert^2 - \left([\nabla J]\left(r \frac{x}{\vert x \vert}\right)\cdot \frac{x}{\vert x \vert}\right)^2 \right)
\]
\end{lem}

\begin{lem}\label{lem:x_minus_2}
Let $N\geq 3$ and let $\phi \in C(\RR^N\backslash\{0\})$ be any function that depends only on $x/\vert x\vert$, then
\[
	\int_{\RR^N} \vert x\vert^{-2} \gamma^{\otimes N}(x) \phi\left(\frac{x}{\vert x \vert}\right) dx =  \frac{1}{(N-2)}  \int_{S^{N-1}(1) } \phi(w) \sigma^N_1(dw).
\]
\end{lem}

\begin{proof}
\begin{eqnarray*}
	\int_{\RR^N} \vert x\vert^{-2} \gamma^{\otimes N}(x) \phi\left(\frac{x}{\vert x \vert}\right) dx & = & \int_{S^{N-1}(1) } \phi(w) \sigma^N_1(dw) \vert S^{N-1}(1) \vert \int_{r=0}^\infty r^{N-3} \frac{e^{\frac{-r^2}{2}}}{(2\pi )^{\frac{N}{2}}} dr \\
	& = &  \int_{S^{N-1}(1) } \phi(w) \sigma^N_1(dw) \frac{1}{2\pi} \frac{\vert S^{N-1}(1) \vert}{\vert S^{N-3}(1) \vert} \\
	& =&  \frac{\pi}{(N/2-1)} \frac{1}{2\pi}   \int_{S^{N-1}(1) } \phi(w) \sigma^N_1(dw) \\
	& = & \frac{1}{N-2}\int_{S^{N-1}(1) } \phi(w) \sigma^N_1(dw).
\end{eqnarray*}
Here we used the fact that $\vert S^{N-1}(1) \vert = \frac{2\pi^{\frac{N}{2}}}{\Gamma\left(\frac{N}{2}\right)}$.	
\end{proof}

We are now ready to prove Proposition \ref{prop:fisher-main}.

\begin{proof}[ Proof of Proposition \ref{prop:fisher-main}]
Note that $\check{G}(x) / \gamma^{\otimes N}(x) = G(\sqrt{N} \frac{x}{\vert x \vert })$.\\ Notice that $\nabla G(\sqrt{N} \frac{x}{\vert x \vert }) = \frac{\sqrt{N}}{\vert x \vert} \left. \nabla_S G\right\vert_{\sqrt{N} \frac{x}{\vert x \vert}}$. Thus, we have:

\begin{eqnarray*}
 I(\check{G} \vert \gamma^{\otimes N}) & = & \int_{\RR^N} \gamma^{\otimes N}(x) \frac{N }{\vert x\vert^2}  
\frac{ \left( \nabla_S G \left( \sqrt{N} \frac{x}{\vert x \vert}\right) \right)^2}{G\left(  \sqrt{N} \frac{x}{\vert x \vert} \right)} d^N x\\
	& = & \frac{N}{N-2} \int_{S^{N-1}(\sqrt{N})} \sigma^N(dw) \frac{\left\vert \nabla_S G(w) \right\vert^2}{H(w)}\\
	& = & \frac{N}{N-2}  I(G \vert \sigma^N).
\end{eqnarray*}
Here we used Lemmas \ref{lem:x_minus_2} and \ref{lem:grad_hat_x}. This proves point \ref{item1}.

In order to prove point \ref{item2}, we will resort to Jensen's inequality. Recall that the mapping 
\[
    (x, \vec{y} ) \mapsto \frac{ \vert \vec{y}\vert^2}{x}
\]
is convex, also the following representation formula for $\hat{F}(x)$ (here $\vert x \vert = R$)
\[
	\hat{F}(x) = \int_{r=0}^\infty r^{N-1} \frac{ F}{ \gamma^{\otimes N}}\left(r \frac{x}{\vert x \vert} \right) \gamma^N(r) \vert S^{N-1}(1) \vert dr,
\]
which is obtained from \eqref{eq:F_hat_N_formula}, convex since $
	 \int_{r=0}^\infty r^{N-1} \gamma^N(r) \vert S^{N-1}(1) \vert dr = 1$. The Fisher information of $\check{\hat{F}}$ is given by
\begin{eqnarray*}
	I(\check{\hat{F}} \vert \gamma^{\otimes N}) & = & \int_{\RR^N} \gamma^{\otimes N}(x) \frac{ \vert S^{N-1}(1) \vert^2 \vert \nabla \int_0^\infty r^{N-1} F(r \frac{x}{\vert x\vert}) dr \vert^2} { \vert S^{N-1}(1)\vert \int_0^\infty r^{N-1} F\left( r \frac{x}{\vert x \vert} \right) dr}  \, dx \\	& = & \int_{\RR^N} \gamma^{\otimes N}(x) \frac{ \left\vert \nabla  \int_{r=0}^\infty r^{N-1} \frac{ F}{ \gamma^{\otimes N}} (r \frac{x}{\vert x \vert} )     \gamma^N(r) \vert S^{N-1}(1) \vert dr \right\vert^2}{ \int_{r=0}^\infty r^{N-1} \frac{ F}{ \gamma^{\otimes N}} (r \frac{x}{\vert x \vert} )     \gamma^N(r) \vert S^{N-1}(1) \vert dr} \\
   & \leq & \int_{\RR^N} \gamma^{\otimes N}(x) \int_{r=0}^\infty r^{N-1} \gamma^N(r) \vert S^{N-1}(1) \vert \frac{\left\vert\nabla\left[\frac{ F}{ \gamma^{\otimes N} }(r \frac{x}{\vert x\vert} )\right]\right\vert^2}{\frac{F}{\gamma^{\otimes N}}\left( r \frac{x}{\vert x \vert} \right)}\, dr. 
\end{eqnarray*}
We simplify the last expression by introducing $\rho=F/\gamma^{\otimes N}$ and using the chain rule in Lemma \ref{lem:grad_hat_x} and the identity in Lemma \ref{lem:x_minus_2}. The end result is:
\begin{align*}
I(\check{\hat{F}} \vert \gamma^{\otimes N})
&\leq \int_{\RR^N}  \gamma^{\otimes N}(x) \int_{r=0}^\infty r^{N-1}\frac{r^2}{\vert x \vert^{2}} \gamma^N(r) \vert S^{N-1}(1) \vert \frac{\left\vert [\nabla\frac{ F}{ \gamma^{\otimes N}}]\left( r \frac{x}{\vert x \vert} \right)\right\vert^2-  
([\nabla\frac{ F}{ \gamma^{\otimes N}}]\left( r \frac{x}{\vert x \vert} \right)\cdot \frac{x}{\vert x\vert})^2] }{\frac{F}{\gamma^{\otimes N}}\left( r \frac{x}{\vert x \vert} \right)}\, dr \\  
&= \frac{1}{N-2} \int_{\RR^N} \vert y\vert^2 \gamma^{\otimes N}(y) \frac{ \left\vert \nabla \rho \right\vert^2(y)- (\nabla\rho(y). y)^2/\vert y \vert^2 }{ \rho(y) } dy. \qedhere
\end{align*}
\end{proof}
The proof of Proposition \ref{prop:I_N_inequality} will require the vanishing of the boundary terms arising from an integration by parts. The following lemma serves that purpose.

\begin{lem}\label{lem:limit_of_x_f_x}
Let $\int_\RR f(x)\, dx=1, \int_\RR x^2 f(x) dx=1$ and $I(f \vert \gamma)< \infty$. Then $ x f'(x) \in L^1(\RR)$ and $\lim_{x\to \pm \infty} x f(x) = 0$.
\end{lem}
\begin{proof} We first recall that
\[ \frac{[(f/\gamma)'(x)]^2}{(f/\gamma)(x)} \gamma(x) = \frac{ \left(f'(x)+ x f(x)\right)^2}{f(x)}.\]
Thus $I(f\vert \gamma) = \int_\RR  \frac{ \left(f'(x)+ x f(x)\right)^2}{f(x)} dx$. To show the integrability of $x f'(x)$ on $\RR$, we write: for any $A > 0$,
\begin{eqnarray*}
	\int_\RR \vert x f'(x) \vert dx & \leq &  \int_\RR \vert x [f'(x)+ x f ]\vert dx + \int_\RR x^2 f(x)\, dx \\
	& = &	\int_\RR \vert x \vert \sqrt{f(x)} \frac{ \vert f'(x)+ x f \vert }{\sqrt{f(x)}}\vert dx + 1\\
	& \leq &  \sqrt{ I(f \vert \gamma) } + 1.	
\end{eqnarray*}
We now use integration by parts on $xf'(x)$ to show that  $\lim_{A\to \infty} A f(A)$ exists,  and deduce that $\lim_{A\to \infty} A f (A) = 0$.
$ \lim_{A\to \infty} A [f(A)] = \int_0^\infty x f'(x) dx + \int_0^\infty f(x) dx$.  Thus $\lim_{A\to \infty} A f(A) = 0$ since $f(x)$ is integrable.
The claim for $\lim_{A\to \infty} Af(-A) =0$ is similarly obtained.
\end{proof}

We are now ready to prove Proposition \ref{prop:I_N_inequality}.

\begin{proof}[Proof of Proposition \ref{prop:I_N_inequality}]
\textbf{Case $1$: $N\geq 3$}
We plug in $f^{\otimes N}$ for $F$ in Proposition \ref{prop:fisher-main}, then $\rho$ can be expressed as

\[
\rho(x) =\frac{f(x_1) \dots f(x_N)}{\gamma(x_1) \dots \gamma(x_N)} =: \rho_1^{\otimes N}.
\]

Due to the product structure of $\rho$, it is convenient to write its gradient as
\[
	\nabla \rho(x_1, \dots, x_N) = \rho(x_1, \dots, x_N) \left( \frac{\rho_1'(x_1)}{\rho_1(x_1)}, \ldots , \frac{\rho_1'(x_N)}{\rho_1(x_N)}\right).
\]
We first consider the inequality (ii) in Proposition \ref{prop:fisher-main}. In order to justify an integration by parts later on, we fix $A >> 1$ and first take our integrals on the region $I_N :=[-A,A]^N$. Afterwards, we can let $A\rightarrow \infty$. Expanding the sums $\vert x \vert^2 = x_1^2 + \dots + x_N^2$ and $\vert \nabla \rho\vert^2 = \rho(x)^2 \sum_{j=1}^N \left( \frac{\rho_1'(x_j)}{\rho_1(x_j)} \right)^2$, we have:
\[
\begin{split}
 \int_{I_N} \vert x\vert^2 \gamma^{\otimes N}(x) \frac{ \left\vert \nabla \rho(x) \right\vert^2}{\rho(x)} d^N x  =  N \int_{-A}^A \gamma(x_1) x_1^2 \frac{\rho_1'(x_1)^2}{\rho_1(x_1)} dx_1 \left( \int_{-A}^A f(x_2) dx_2\right)^{N-1}\\
 + N(N-1) \int_{-A}^A x^2 f(x) \,dx \left(\int_{-A}^{A} f(x_2) \, dx_2\right)^{N-2}\int_{-A}^{A} \gamma(x_1)  \frac{\rho_1'(x_1)^2}{\rho_1(x_1)} dx_1 
\end{split}
\]
similarly, the second term in (ii) gives us the following.
\[
\begin{split}
-\int_{I_N} \gamma^{\otimes N}(x) \frac{(\nabla\rho(x). x)^2}{\rho( x)}  dx = - N \int_{-A}^A x_1^2 \gamma(x_1) \frac{\rho_1'(x_1)^2}{\rho_1(x_1)} dx_1 \left( \int_{-A}^A f(x_2) dx_2 \right)^{N-1}\\
     - N(N-1) \left(\int_{-A}^A \gamma(x_1) x_1 \rho_1'(x_1) \right)^2 \left( \int_{-A}^A f(x_2) dx_2 \right)^{N-2}.
\end{split}
\]
This last terms simplifies since an integration by parts, together with Lemma \ref{lem:limit_of_x_f_x}, shows that 
\[
	- \int_{-A}^A \gamma(x_1) x_1 \rho_1'(x_1) dx_1 =   A[f(A)+ f(-A)] +  \int_{-A}^{A} \left(\gamma(x_1)\rho_1(x_1) - x_1^2\gamma(x_1)   \rho_1(x_1) \right) dx_1.
\]
We will set the above to be $\delta_A$. Lemma \ref{lem:limit_of_x_f_x} implies that $\lim_{A\rightarrow \infty} \delta_A = 1 -  \int_\RR x_1^2 f(x_1) dx_1 = 0$.  

Therefore, 
\begin{align*}
I(\check{\hat{f}}^N \vert \gamma^{\otimes N}) & \leq \lim_{A \rightarrow \infty}\frac{1}{N-2} \int_{I_N} \gamma^{\otimes N}(x) \frac{  \vert x\vert^2 \left\vert \nabla \rho(x) \right\vert^2- (\nabla\rho(x). x)^2 }{ \rho(x) } dx  \\
 & =  \frac{ N(N-1)}{ (N-2) } \lim_{A \rightarrow \infty} \left[ \int_{-A}^A x^2 f(x) dx  \left(\int_{-A}^{A} f(x_2) \, dx_2\right)^{N-2} \ast
  \right. \\
 & \quad \left.\int_{-A}^{A} \gamma(x_1)  \frac{\rho_1'(x_1)^2}{\rho_1(x_1)} dx_1 - \delta_A^2 \left( \int_{-A}^A f(x_2) dx_2 \right)^{N-2}\right]\\
   & = N \frac{N-1}{N-2} I(f \vert \gamma).
\end{align*}
Finally, from 1 in Proposition \ref{prop:fisher-main}, the coefficient of $I(\hat{f}^N\vert \sigma^N)$ is simply $\frac{N-2}{N} \times \mbox{ (the coefficient of $I(\check{\hat{f}}^N\vert \gamma^{\otimes N})$ )}$ which is $N\left( 1 - \frac{1}{N}\right)$.
This completes the proof for the case $N\geq 3$.

\textbf{Case $2$: $N=2$}\\
In this case, we cannot use Proposition \ref{prop:fisher-main}. Instead, we treat with
the rescaled state $\hat{f}^2$ directly. Note that we have the formula: $\hat{f}^2(\sqrt{2}\cos\theta, \sqrt{2 }\sin\theta) = 2\pi \int_0^\infty r f(r\cos\theta) f(r\sin\theta) \, dr$,
which, using $\rho(x) = f(x)/ \gamma(x)$, can be rewritten as
\[
		\hat{f}^2(\sqrt{2}\cos\theta, \sqrt{2}\sin\theta) =  \int_0^\infty r e^{-r^2/2} \rho(r\cos\theta) \rho(r\sin\theta)\, dr.
\]
The relative Fisher information for $\hat{f}^2$ is given by $ I(\hat{f}^2 \vert \sigma^2) = \int_0^{2\pi} \frac{d\theta}{2\pi} \frac{ \left( \frac{1}{\sqrt{2}}\frac{d}{d\theta} \hat{f}^2 \right)^2}{\hat{f}^2}$.
Since $r e^{-r^2/2}$ is a weight and the mapping $u \mapsto \frac{u'^2}{u}$ is convex, Jensen's inequality can be applied. This gives:
\begin{eqnarray*}
 I(\hat{f}^2 \vert \sigma^2) & = & \frac{1}{4 \pi } \int_0^{2\pi} d\theta \frac{ \left( \frac{d}{d\theta} \hat{f}^2 \right)^2}{\hat{f}^2} \leq  \frac{1}{4 \pi} \int_0^{2\pi} d\theta \int_{r=0}^\infty dr\, r e^{-r^2/2} \frac{ \left( \frac{d}{d\theta} \rho(r\cos\theta)\rho(r\sin\theta) \right)^2}{\rho(r\cos\theta)\rho(r\sin\theta)}\\
		& = &  \frac{1}{4\pi} \int_0^{2\pi} d\theta \int_{r=0}^\infty dr\, r e^{-r^2/2} \left[ 2\frac{\rho'(r\cos\theta)^2}{\rho(r\cos\theta)} \rho(r\sin\theta)(r\sin\theta)^2 \right. \\ &   & \left. - 2 \rho'(r\cos\theta) \rho'(r\sin\theta) r\cos\theta r\sin\theta	\right]\\	& = & \int_{x\in \RR} \gamma(x) \frac{\rho'(x)^2}{\rho(x)}dx \int_{y\in \RR} \gamma(y) \rho(y) y^2 dy - \left(\int_\RR \rho'(x) \gamma(x) x dx \right)^2\\
		& = & I(f \vert \gamma),
\end{eqnarray*}
where we used the fact that $\int_\RR f(y) y^2 dy = 1$, and that  $\int_\RR \rho'(x) \gamma(x) x dx = 0$ which follows from integration by parts.
\end{proof}

\section{Conclusion and open questions}
\label{sec:conclusion}

In this article, we considered the rescaled states $\hat{f}^N$, defined as the push-forward of $f^{\otimes N}$ by the mapping $x \mapsto \sqrt{N} x / \vert x \vert \in S^{N-1}$, and established their chaoticity quantitatively, in the sense of Kac, using the Wasserstein metrics, in the sense of entropy, Fisher-information, and in the sense of $L^1$ for $f \in \ALip(r) \subset L^1(\RR)$. The rescaled states provide sequences of measures on Kac's sphere which are chaotic to $f$ when $f$ is not necessarily in $L^1 \cap L^p$ for some $p>1$.

Many interesting questions remain unanswered. A first set of questions concerns the set $\ALip(r)$ and $L^1$ chaoticity. We recall from Lemma \ref{lem:Besov} that $\ALip(r)$ is a subset of a Besov space. Can the $L^1$ chaoticity result in Theorem \ref{thm:POC_L1} be extended to all $f$ in this Besov space? Also, can the space $\ALip$ be characterized further?
    
Another set of questions we left open concerns quantitative entropic chaoticity. The structure of the rescaled states allowed us to prove entropic chaoticity relatively easily. But in order to show quantitative entropic chaoticity, we used Theorem \ref{thm:H&M_JFA2014} (an improvement to \cite[Theorem 4.17]{hauray-mischler2014}), whose proof relies on a comparison with conditioned tensor products states $[f^{\otimes N}]_N$. This required $f$ to have a finite relative Fisher information and that $f \in  \mathcal{P}_{4+\epsilon}(\RR) \cap L^p(\RR)$. It would be interesting to get rid of these restrictions. Also, avoiding the use of Theorem \ref{thm:H&M_JFA2014} would make our quantitative entropic chaos result independent of the conditioned tensor products. Moreover, since the rescaled tensor product $\hat{f}^N$ requires less assumptions on $f$ than $[f^{\otimes N}]_N$ does, one could adapt the proof of Theorem \ref{thm:H&M_JFA2014} in order to obtain quantitative entropic chaos rates valid for a broader class of sequences, by comparing them directly with $\hat{f}^N$.
    
    We left open the problem of quantitative rates of
    Fisher-information chaos for the rescaled tensor product states.

Finally, we mention the following question. Do we have the limit:
        \[
            \lim_{N\to \infty} H \left(  f^{\otimes k} \, \middle\vert \, \Pi_k \hat{f}^N \right) =0
        \]
under some conditions? Stating the analogous question for $H ( \Pi_k \hat{f}^N \vert f^{\otimes k})$ requires precaution. Even if $f\ll\gamma$, we can have $\Pi_k \hat{f}^N \centernot{\ll} f^{\otimes k}$. This occurs for example when $f([-a,a])=0$ for some $a>0$, but $f$ is not compactly supported.

\section{Appendix}

We now aim to prove Theorem \ref{thm:semi-extens}. As in \cite[Theorem 12]{carlen-carvalho-leroux-loss-villani2010}, the proof uses the following result:

\begin{lem}[Legendre representation of the entropy]
\label{lem:Legendre}
Let $E$ be a locally compact metric space, let $\mu, \nu \in \P(E)$. Then
\begin{equation}
    \label{eq:Legendre}
    H(\mu \mid \nu)
    = \sup \left\{ \int \phi d\mu - \log \int e^{\phi} d\nu : \text{$\phi \in C_b(E)$ and Lipschitz} \right\}.
\end{equation}
Moreover, one can restrict the supremum to functions satisfying $\int e^{\phi} d\nu = 1$.
\end{lem}

\begin{proof}
Equation \eqref{eq:Legendre}, without the Lipschitz condition, is part of the folklore; see for instance \cite[Theorem B.2]{lott-villani2009} for a proof in the case of $E$ compact. From there, the restriction to Lipschitz functions is straightforward: since $E$ is locally compact, one can restrict the supremum to continuous and compactly supported functions, which can further be approximated by Lipschitz functions in the supremum norm so that both integrals in \eqref{eq:Legendre} are close to the originals. We omit the details.
\end{proof}

\begin{proof}[Proof of Theorem \ref{thm:semi-extens}]
We follow the proof of \cite[Theorem 12]{carlen-carvalho-leroux-loss-villani2010}. Fix $\epsilon>0$. From Lemma \ref{lem:Legendre}, there exists $\phi \in C_b(\RR)$ $L$-Lipschitz such that $\int_\RR e^{\phi(x)} \gamma(x) dx = 1$ and
\begin{equation}
    \label{eq:thm:semi-extens_int-phi-f}
    \int_\RR \phi(x) f(dx)
    \geq H(f \mid \gamma) - \epsilon.
\end{equation}
Define $\Phi \in C_b(\RR^N)$ as $\Phi(x) = \phi(x_1) + \cdots + \phi(x_N)$. Let $Z \sim \gamma^{\otimes N}$, thus $\hat{Z} \sim \sigma^N$. From \eqref{eq:Legendre} and symmetry, we obtain
\begin{align}
    \frac{1}{N} H(F^N \mid \sigma^N)
    &\geq \frac{1}{N} \int_{S^{N-1}} \Phi(x) F^N(dx) - \frac{1}{N} \log \int_{S^{N-1}} e^{\Phi(x)} \sigma^N(dx)
    \notag \\
    &= \int_{S^{N-1}} \phi(x_1) F^N(dx) - \frac{1}{N} \log \EE[e^{\Phi(\hat{Z})}]
    \label{eq:thm:semi-extens_HFN}.
\end{align}
The idea is to replace $\Phi(\hat{Z})$ by $\Phi(Z)$, so we now control their difference. Let $Q = \frac{1}{N} \sum_i Z_i^2$, thus $\hat{Z}_i = Q^{-1/2}Z_i$. Since $\phi$ is $L$-Lipschitz, the same argument that leads to \eqref{eq:proof_thm_chaos_rescaled_measures_EQN} gives
\[
    \vert \Phi(\hat{Z}) - \Phi(Z)\vert 
    \leq \sum_{i=1}^N \vert \phi(Q^{-1/2}Z_i) - \phi(Z_i)\vert 
    \leq L N \vert Q-1\vert .
\]
Consider the event $A_N = \{ \vert Q-1\vert  \leq N^{-1/2}\}$. Clearly,
\[
    \EE[\ind_{A_N} e^{\Phi(\hat{Z})}]
    \leq \EE[\ind_{A_N} e^{\Phi(Z)} e^{LN \vert Q-1\vert }]
    \leq e^{LN^{1/2}},
\]
where we have used that $\EE[e^{\Phi(Z)}] = 1$, because $\EE[e^{\phi(Z_i)}] = 1$. Now, since $\hat{Z}$ and $Q$ are independent, we have $\EE[\ind_{A_N} e^{\Phi(\hat{Z})}] = \PP(A_N) \EE[e^{\Phi(\hat{Z})}]$, thus
\[
\frac{1}{N} \log \EE[e^{\Phi(\hat{Z})}]
= \frac{1}{N} \log \frac{\EE[\ind_{A_N} e^{\Phi(\hat{Z})}]}{\PP(A_N)}
\leq \frac{L}{N^{1/2}} - \frac{1}{N} \log \PP(A_N).
\]
By the Central Limit Theorem applied to the sequence $Z_1^2,Z_2^2,\ldots$, we know that $\PP(A_N)$ converges to some strictly positive quantity. From \eqref{eq:thm:semi-extens_HFN}, we thus have
\begin{align*}
    \liminf_{N\to\infty} \frac{1}{N} H(F^N \mid \sigma^N)
    &\geq \liminf_{N\to \infty} \left\{ \int_{S^{N-1}} \phi(x_1) F^N(dx)
    - \frac{L}{N^{1/2}} + \frac{1}{N} \log \PP(A_N)\right\} \\
    &= \int_\RR \phi(x) f(dx),
\end{align*}
where we have used that $\Pi_1 F^N$ converges weakly to $f$. Using \eqref{eq:thm:semi-extens_int-phi-f} and letting $\epsilon \to 0$, the conclusion follows.
\end{proof}

\begin{proof}[Proof of Lemma \ref{lem:conditioned-product-improvement}]
Here, we describe the proof of \cite[Theorem 4.13]{hauray-mischler2014} and state the adjustments needed to prove \eqref{eq:entropic_chaos_conditioned_estates}.
We recall that for $f \in L^1(\RR)$, $[f^{\otimes N}]_N$ denotes the conditioned tensor product state defined in Def. \ref{def:conditioned-states}. If $f \in L^p(\RR)$ for some $p>1$ and $\int_\RR \vert x \vert^{k} f(x) \,dx < \infty$ for some $k>2$, then $\{[f^{\otimes N}(x)]_N\}_N$ is entropically chaotic to $f$. When $k\geq 4$, this was proven in \cite{carlen-carvalho-leroux-loss-villani2010} without any rates. When $k \geq 6$, entropic chaoticity with a rate of $N^{-1/2}$ was proven in \cite[Theorem $4.13$]{hauray-mischler2014}. Finally, when $2<k<4$, entropic chaoticity of the family $\{[f^{\otimes N}(x)]_N\}$ was proven in \cite{carrapatoso-einav2013} without rates of convergence. The quantitative result in \cite[Theorem $4.13$]{hauray-mischler2014} can be extended to the case $k>4$. We mention the technical changes necessary to relax the finite $6^{th}$ moment requirement.
\begin{itemize}
    \item The Fourier based Lemma $4.8$ in \cite{hauray-mischler2014} can be modified to densities $g$ having only $M_{2+r}$ moments for some $r \in (0,1]$. The new statement becomes:
    
    Let $g \in \mathcal{P}_{2+r}(\RR) \cap L^p(\RR)$ for some $p>1$,\\ $\int_\RR x g(x) dx =0, \int_\RR x\otimes x g(x) \, dx = \Id, \int_\RR \vert x\vert^{2+r} g(x) \, dx = M_{2+r}$. Then:
\begin{enumerate}
    \item $\exists \delta>0$ such that $\vert \xi\vert \leq \delta \rightarrow \vert \hat{g}(\xi)\vert \leq e^{-\vert \xi\vert^2/4}$.
    \item $\forall \delta>0, \exists k=k(M_{2+r},r, p, \Vert g\Vert_p, \delta) \in (0,1)$ such that
    \[
        \sup_{\vert \xi\vert \geq \delta} \vert \hat{g}(\xi)\vert \leq k(\delta)
    \]
\end{enumerate}
The first claim above follows from the idea that if $g$ has a finite $2+r$ moment, then
\[
    \hat{g}(\xi) = 1 - \frac{\vert \xi\vert^2}{8} + R(\xi), \quad \vert R(\xi) \vert \leq \frac{h_r}{(r+1)(r+2)} M_{2+r}(g) \vert \xi\vert^{2+r},
\]
where $h_r = \sup_{\theta\neq 0} \vert \theta\vert^{-r} \vert e^{-i\theta}-1\vert$.

The second claim follows from \cite[Theorem 2.7(i)]{carlen-carvalho-leroux-loss-villani2010}, and the observation that $\int g \log g < \infty$, whenever $g \in L^p$.

\item Using the above result, the local central limit theorem in \cite[Theorem 4.6]{hauray-mischler2014} can be extended to $g \in \mathcal{P}_{2+r}(\RR) \cap L^p(\RR)$, $r\in (0,1]$, $p \in (1, \infty]$. So that if $g_N(x)$ is the iterated and renormalized convolution
\[
    g_N(x) = \sqrt{N} g^{(\ast N)}(\sqrt{N} x),    
\]
then $\exists N=N(p)$ and $C_{BE}=C(p,r, M_{2+r}(g), \Vert g \Vert_p)$ such that 
\[
    \forall N \geq N(p), \quad \Vert g_N - \gamma \Vert_\infty \leq \frac{C_{BE}}{N^{r/2}}.
\]
This can be proved exactly as in \cite[Theorem 4.6]{hauray-mischler2014}, with the additional observation that
\[\sup_{\xi \neq 0} \frac{\left\vert \hat{g}_N(\xi) -\hat{\gamma}_N(\xi)\right\vert}{\vert \xi\vert^{2+r}} < \infty.\]

\item The above observations can pass on to \cite[Theorem 4.9]{hauray-mischler2014}, with a remainder term $R_N(x)/ N^{r/2}$. with $R_N \in L^\infty$.

\item It remains to show that even when $f$ has only $4 + s$ moments, the quantities $\theta_{N,1}$ defined by eq. $(4.18)$ in \cite{hauray-mischler2014} stay close to $1$. Using $\vert e^{-x^2}-1 \vert \leq C_\alpha \vert x\vert ^{2\alpha}$ for any $\alpha \in (0,1]$, one can show that Eq. $(4.20)$ in \cite{hauray-mischler2014} can be replaced by
\[
    \vert \theta_{N,l}(v) -1 \vert \leq \frac{C l^2}{N^{1/2}} + O(N^{-s/4}) + C_\alpha \frac{\vert  v\vert^{4\alpha}}{ N^\alpha} \mathbf{1}_{[N^u, \infty)}(\vert V\vert)
\]
provided $\alpha (1- 4 u) = s/4$. Choosing $u= \frac{s}{s+2}$ and $\alpha= \frac{s(s+2)}{8}$ allows us to carry on the same decomposition of $f$ used in the proof of Theorem $4.13$ in \cite{hauray-mischler2014} and to arrive at \eqref{eq:entropic_chaos_conditioned_estates}.
\end{itemize}
We emphasize that the condition $\int_\RR v f(v) \,dv=0$ mentioned is \cite[Theorem $4.13$]{hauray-mischler2014} is not required.
\end{proof}

\begin{proof}[Sketch of the proof of Lemma \ref{lem:conditioned-product-bounded-fisher-information}]
Starting with the formula
\[
    I([f^{\otimes N}]_N \vert \sigma^N) = \int_{S^{N-1}} \left\vert \nabla_S \log h_N(w) \right\vert [f^{\otimes N}]_N( dw)
\]
where $h_N$ is the density of $[f^{\otimes N}]_N$ with respect to $\sigma^N$, given by
\[
\frac{d [f^{\otimes N}]_N}{ d \sigma^N} (w)
= \frac{f^{\otimes N}(w)}{ \int_{S^{N-1}} f^{\otimes N}(y) \sigma^N(dy)}.
\]
We note that we can replace $h_N(w)$ by $\frac{ f^{\otimes N}(w)}{ \gamma^{\otimes N}(w) }$ since $\gamma^{\otimes N}$ is constant on $S^{N-1}$.

Using $\vert \nabla_SG (v) \vert^2 \leq \vert \nabla G(v) \vert^2$ when $G$ is defined on $\RR^N$, we obtain (see Eq. (4.24) in \cite{hauray-mischler2014}):
\[
    \frac{1}{N} I([f^{\otimes N}]_N \vert \sigma^N) \leq  I(f \vert \gamma) + \int_\RR \left\vert \frac{\nabla f(v)}{f(v)} + v \right\vert^2 (\theta_{N,1}(v)-1) f(v) \, dv,
\]
where the product $\theta_{N,1}(v) f(v)$ equals the first marginal of $[f^{\otimes N}]_N$ and, $\theta_{N,1}(v)$ satisfies $\vert \theta_{N,1}(v) \vert \leq C$ uniformly in $N$ (see \cite[Equation (4.20)]{hauray-mischler2014}). Thus, $\frac{1}{N} I([f^{\otimes N}]_N\vert \sigma^N) \leq (1+ C) I(f\vert \gamma)$.
\end{proof}


\end{document}